\documentclass{amsart}
\usepackage{amssymb, amsmath, tikz-cd, amscd}
\usepackage{stmaryrd}
\usepackage{xcolor}

\title[Variation of the canonical height]{Variation of the canonical height in a family of polarized dynamical systems}
\author{Patrick Ingram}
\address{York University, Toronto, Canada}
\email{pingram@yorku.ca}

\renewcommand{\epsilon}{\varepsilon}
\renewcommand{\phi}{\varphi}

\newcommand{\PP}{\mathbb{P}}
\newcommand{\ZZ}{\mathbb{Z}}

\newcommand{\RR}{\mathbb{R}}
\newcommand{\QQ}{\mathbb{Q}}
\renewcommand{\AA}{\mathbb{A}}

\newcommand{\ord}{\operatorname{ord}}

\newcommand{\hgeom}{h^{\mathrm{geom}}}
\newcommand{\harith}{h^{\mathrm{arith}}}
\newcommand{\href}{h^{\mathrm{ref}}}
\newcommand{\htot}{h^{\mathrm{total}}}

\newtheorem{lemma}{Lemma}

\newtheorem{theorem}[lemma]{Theorem}
\newtheorem{corollary}[lemma]{Corollary}

\theoremstyle{definition}

\begin{document}
\begin{abstract}
Call and Silverman introduced the canonical height $\hat{h}_{f}$ associated to an endomorphism $f:X\to X$ of a projective variety and an ample  $L\in\operatorname{Pic}(X)$ satisfying $f^*L\cong L^{\otimes d}$ for some $d\geq 2$. They also presented an asymptotic for the variation of this height in a family over the one-dimensional base $B$ in terms of the height on the generic fibre and the height of the parameter, namely
\[\hat{h}_{f_t}(P_t)=\hat{h}_{f}(P)h_B(t)+o(h_B(t)),\]
where $o(x)/x\to 0$ as $x\to\infty$. Here we save a power in the error term, giving an effective estimate
\[\hat{h}_{f_t}(P_t)=\hat{h}_{f}(P)h_B(t)+O\left(h_B(t)^{2/3}\right)\]
in general, and
\[\hat{h}_{f_t}(P_t)=\hat{h}_{f}(P)h_B(t)+O\left(h_B(t)^{1/2}\right)\] when $B$ is rational.
 As a corollary, we give an explicit bound on the height of parameters $t\in B$ for which $P_t$ is preperiodic for $f_t$, in the case that $X=\PP^N$ and $B=\PP^1$.
\end{abstract}
\maketitle

\section{Introduction}

The canonical height associated to a polarized dynamical system is a fundamental tool in arithmetic dynamics. To a projective variety $X$ with a line bundle $L$, we may associate a Weil height $h$ which depends on various choices of defining equations. If $f$ is an endomorphism of $X$ with $f^*L\cong L^{\otimes d}$ for some $d\geq 2$, and $L$ is ample, then we say that $(f, X, L)$ is \emph{polarized}, and we may associate a \emph{canonical height} $\hat{h}_{f}$ satisfying
\[\hat{h}_{f}(f(P))=d \hat{h}_{f}(P)\quad\text{ and }\quad \hat{h}_{f}(P)=h(P)+O(1),\]
where $h$ is any Weil height on $X$ with respect to $L$.
This construction is due to Call and Silverman~\cite{MR1255693} in the general case, but see also work of  Denis~\cite{MR1183402}, and Zhang~\cite{MR1311351}, all building on classical work of N\'{e}ron and Tate.
Over a number field, sets of bounded ample height are finite, and so we may deduce, as a quick application, the finiteness of the set of rational preperiodic points for $f$.

 Now let $(f, X, L)$ be a family of polarized dynamical systems over a curve $B$, and let $P:B\to  X$. On all but finitely many fibres $(f_t, X_t, L_t)$ we may construct a canonical height, and Call and Silverman also showed~\cite[Theorem~4.1]{MR1255693} that
\begin{equation}\label{eq:csvar}\hat{h}_{f_t}(P_t)=\hat{h}_{f}(P)h_B(t)+o(h_B(t)),\end{equation}
 where $\hat{h}_{f}(P)$ is the canonical height on the generic fibre,  where $h_B$ is a Weil height on $B$ with respect to any divisor of degree one, and where $o(x)/x\to 0$ as $x\to\infty$. Naturally, Call and Silverman ask~\cite[p.~184]{MR1255693} whether one might improve upon~\eqref{eq:csvar}.

In the case where $X=E$ is an elliptic curve (and so the canonical height is the N\'{e}ron-Tate height), Tate~\cite{MR692114} improved this estimate by establishing a result of the form
\[\hat{h}_{f_t}(P_t)=\hat{h}_{f}(P)h_B(t)+O(h_B(t)^{1/2}),\]
with a further improvement to $O(1)$ in the error when $B$ is rational. Note that this is, in some sense, the best possible error term for an arbitrary height function $h_B$.

In the case $X=\PP^1$ and with $f$ a family of polynomials, the author~\cite{MR3181564} proved analogous results, and this was extended to endomorphisms of $\PP^N$ with fixed behaviour on a totally invariant hyperplane~\cite{MR3436155}. Still in the case $X=\PP^1$,  Ghioca, Hsia, and Tucker~\cite{MR3335283} proved a similar result for rational functions with a super-attracting fixed point in the case $\hat{h}_f(P)\neq 0$. Generalizing work of Ghioca and Mavraki~\cite{MR3158237}, Mavraki and Ye~\cite{my} established the same for rational functions in the case $X=B=\PP^1$, but under the additional hypothesis that the pair $(f, P)$ is \emph{quasi-adelic}, a condition which is not known to always hold. Related results also exist for H\'{e}non maps~\cite{MR3180596}, Drinfeld modules~\cite{MR3257554}, and dynamical correspondences~\cite{MR3885169}. As well, the aforementioned result of Tate has been sharpened by DeMarco and Mavraki~\cite{dm} to show that $\hat{h}_{E_t}(P_t)$ is precisely a height induced by an adelically metrized line bundle on the base,  while the case of single-variables polynomials from~\cite{MR3181564} was subsequently similarly strengthened by Favre and Gauthier~\cite{MR3864204}.

Our main result here is a weaker  savings in the error term of the estimate, but in the general case.
\begin{theorem}\label{th:main}
Let $(f, X, L)$ be a family of polarized dynamical systems over the curve $B$, defined over a number field $K$, and let $h_B$ be a degree-one Weil height on $B$. Then  we have
\[\hat{h}_{f_t}(P_t)=\hat{h}_{f}(P)h_B(t)+O\left(h_B(t)^{\frac{2}{3}}\right)\]
as $h_B(t)\to\infty$.
When $B$ is rational we have the further improvement
\[\hat{h}_{f_t}(P_t)=\hat{h}_{f}(P)h_B(t)+O\left(h_B(t)^{\frac{1}{2}}\right).\]
\end{theorem}


As already remarked in~\cite{MR1255693}, the asymptotic \eqref{eq:csvar} ensures that when $\hat{h}_f(P)>0$, the $t\in B(\overline{K})$ for which $P_t$ is preperiodic for $f_t$ form a set of bounded height. Indeed, the asymptotic obtained has the form
\[\left|\hat{h}_{f_t}(P_t)-\hat{h}_f(P)h_B(t)\right|\leq \epsilon h_B(t)+O_\epsilon(1)\]
for any $\epsilon>0$, where the implied constant is in-principle computable; one need only compute the implied constant for some $0<\epsilon<\hat{h}_f(P)$. 
In practice, though, this depends on some lengthy computations for each example.
In the case of endomorphisms of $X=\PP^N$ over $B=\PP^1$,  we can produce a completely explicit bound depending on relatively natural measures of complexity. For $P_0, ..., P_N\in K[t]$ with no common factor, set
\[\hgeom([P_0:\cdots :P_N])=\max\deg(P_i),\]
let $\harith(P)$ be the height of the projective tuple of coefficients of all of the $P_i$, and set \[\htot(P)=\harith(P)+\hgeom(P).\] For $f:\PP^N\to \PP^N$ defined over $K(t)$, we   define $\htot(f)$ by identifying $f$ with its tuple of coefficients.

\begin{corollary}\label{cor:main}
Let $f:\PP^N\to\PP^N$ be defined over $\PP^1_K$, and let $P\in \PP^N(\PP^1_K)$ such that $\hat{h}_f(P)\neq 0$. If $P_t$ is preperiodic for $f_t$, then
\[h(t)\leq C\max\left\{\frac{(\htot(P)+\htot(f)+1)^4}{\hat{h}_f(P)^2}, \htot(f)+1\right\}\]
where $C$ is an explicit constant depending just on $\deg(f)$ and $N$.
\end{corollary}

The question of exactly when we have $\hat{h}_f(P)=0$ is  more subtle in the function field context (i.e., on the generic fibre) than it is over a number field. In the situation $X=\PP^1$, it follows from results of Benedetto~\cite{MR2202175} and Baker~\cite{MR2492995} that $\hat{h}_f(P)=0$ only when $P$ is preperiodic for $f$ (in which case $P_t$ is always preperiodic for $f_t$), or the pair $(f, P)$ is isomorphic over some function field extension to a constant family (in which case $P_t$ either is or isn't preperiodic for $f_t$, independent of $t$). Gauthier and Vigny~\cite{gv} have extended this result to families of polarized dynamical systems in general, showing that if $X$ has no periodic isotrivial subvariety of positive dimension, and $\hat{h}_f(P)=0$, then $P$ is preperiodic for $f$.

Call and Silverman actually defined canonical heights in an even more general setting. If $f:X\to X$ and $L\in \operatorname{Pic}(X)\otimes \RR$ (not necessarily ample) satisfies $f^*L\cong L^{\otimes \alpha}$ with $\alpha>1$ real, the same construction of a canonical height with respect to $L$ goes through, and the same asymptotic~\eqref{eq:csvar} holds. This is a strictly more general setting: Silverman~\cite{MR1115546} has constructed examples of automorphisms of certain K3 surfaces admitting $L\in \operatorname{Pic}(X)\otimes \RR$ with $f^*L\cong L^{\otimes (7+4\sqrt{3})}$. In this context we obtain a slightly weaker result, restricted to the simpler case $B=\PP^1$.
\begin{theorem}\label{th:k3}
Let $X$ be a family of Wehler K3 surfaces defined over $\PP^1_K$, and let $\hat{h}^{\pm}$ be the canonical heights defined by Silverman~\cite{MR1115546}. Then for any $P\in X(\PP^1_K)$ and any $\epsilon>0$, we have
\[\hat{h}^{\pm}_t(P_t)=\hat{h}^\pm(P)h(t)+O\left(h(t)^{\frac{1}{2}+\epsilon}\right).\]	
\end{theorem}
Theorem~\ref{th:k3} is really a corollary to the more general (and more technical) Theorem~\ref{th:matrix} in Section~\ref{sec:k3}, on endomorphisms $f$ of projective   varieties and systems of line bundles $L_i$ satisfying $f^*L_i\cong \bigotimes_{j=1}^r L_j^{\otimes A_{ij}}$, for some integers $A_{ij}$.

Since our proof is a technical refinement of the argument of Call and Silverman~\cite{MR1255693} (tracing back to an earlier result of Silverman~\cite{MR703488}), we sketch that argument in the context of Theorem~\ref{th:main} before outlining our modification. The proof in \cite{MR1255693} proceeds as follows. First,
\begin{align}
	\left|\hat{h}_{f_t}(P_t)-\hat{h}_{f}(P)h_B(t)\right|&\leq \left|\hat{h}_{f_t}(P_t)-h_{X, L}(P_t)\right|\label{eq:cs1}\\
	&\quad + \left|h_{X, L}(P_t)-\deg(P^*L)h_B(t)\right|\label{eq:cs2}\\
	&\quad + \left|\deg(P^*L)-\hat{h}_{f}(P)\right|h_B(t)\label{eq:cs3}.
\end{align}
On the one hand, the term in~\eqref{eq:cs3} is bounded by estimating the difference between the naive height $\deg(P^*L)$ and the canonical height in the function field, which can be done with no dependence on $P$. On the other hand, it is not hard to show that the difference between the Weil height and canonical height in families is at most linear in the height of the parameter, whence the term on the right in~\eqref{eq:cs1} is at most $O(h_B(t))$, again with no dependence on $P$. Finally, since $h_{X, L}\circ P=h_{B, P^*L}$ and $\deg(P^*L)h_B$ are heights on $B$ relative to divisors of the same degree, the term in~\eqref{eq:cs2} is at most $O(h_B(t)^{1/2})$, with a constant depending on $B$. Consequently, we have
\begin{equation}\label{eq:firststab}\limsup_{h_B(t)\to\infty}\left|\frac{\hat{h}_{f_t}(P_t)}{h_B(t)}-\hat{h}_{f}(P)\right|\leq C,\end{equation}
for some constant $C$ not depending on $P$, since the term corresponding to~\eqref{eq:cs2} vanishes in the limit. But now, using the relation $\hat{h}_f\circ f = d\hat{h}_f$, we may apply this estimate to $f^k(P)$ to obtain
\[\limsup_{h_B(t)\to\infty}\left|\frac{\hat{h}_{f_t}(P_t)}{h_B(t)}-\hat{h}_{f}(P)\right|=\limsup_{h_B(t)\to\infty}\frac{1}{d^k}\left|\frac{\hat{h}_{f_t}(f^k(P)_t)}{h_B(t)}-\hat{h}_{f}(f^k(P))\right|\leq \frac{C}{d^k}.\]
Since $k$ is arbitrary, we may take $C=0$ in~\eqref{eq:firststab}.

Our strategy is to make the bound on the term~\eqref{eq:cs2} more explicit. Specifically,  we employ an explicit elimination of variables to obtain a bound of the form
\begin{equation}\label{eq:mainelim}\left|h_{X, L}(f^k(P)_t)-\deg(f^k(P)^*L)h_B(t)\right|\leq C'd^{3k}\end{equation} for a  specific height function on $B$, 
where $C'$ depends on $f$ and $P$, but not on $k$ or $t$.  Combined with the same estimates used by Call and Silverman for the other terms, we have
\begin{align*}
	\left|\hat{h}_{f_t}(P_t)-\hat{h}_{f}(P)h_B(t)\right|&=d^{-k}\left|\hat{h}_{f_t}(f^k(P)_t)-\hat{h}_{f}(f^k(P))h_B(t)\right|\\
	&\leq C'd^{2k}+C''d^{-k}h_B(t),
	\end{align*}
 with constants independent of both $k$ and $t$. For each $t\in B$ we choose $k$ with $d^{k}\approx h_B(t)^{1/3}$ to obtain an upper bound of scale $h_B(t)^{2/3}$. The further honing when $B=\PP^1$ derives from a sharper version of~\eqref{eq:mainelim} in that context.

\section{Explicit results over $\PP^1$ }\label{sec:rational}

In this section we present tighter results in the case that the base curve $B$ is rational, in which case we take $B=\PP^1_K$ and $h$ to be the standard Weil height on $B$; in other words, our dynamical systems are defined over the function field $K(t)$. We also restrict to the case in which $X=\PP^N_{K(t)}$, although we will see in Lemma~\ref{lem:red} below that this is no loss of generality.

To set notation, let $K$ be a number field, and let $M_K$ be the usual set of absolute values on $K$. We set $n_v=[K_v:\QQ_v]/[K:\QQ]$, and normalize our absolute values so that
\[\sum_{v\in M_K}n_v\log|x|_v=0\]
for $x\neq 0$. We will write
\[\|x_1, ..., x_m\|_v=\max|x_i|_v,\]
and $\log ^+x = \log\max\{1, x\}$, for $x\in \RR$. We also write, for real-valued functions $f$ and $g$, that $f(x)=O(g(x))$ as $x\to\infty$ as long as there exist constants $C$ and $C'$ such that $|f(x)|\leq Cg(x)$ whenever $x\geq C'$.

For $P\in \PP^N(K)$, we set
\[h(P)=\sum_{v\in M_K}n_v\log\|P\|_v\]
as usual.
 We also note that if $L/K$ is a finite extension and $P\in \PP^N(K)$, then $h(P)$ as computed in $\PP^N_L$ agrees with $h(P)$ as computed in $\PP^N_K$, and so the height is well defined over $\overline{K}$. Indeed, all of our estimates below are stable under field extension, and so our results apply at $t$ varies in $\PP^1_{\overline{K}}$.

We identify points $P\in \PP^N(K(t))$ with morphisms $P:\PP^1\to\PP^N$ defined over $K$. Every $P\in \PP^N(K(t))$ can be written as $P=[P_0:\cdots :P_n]$ with $P_i\in K[t]$  with no common factor, uniquely up to multiplication by a non-zero scalar. We write $\hgeom(P)=\deg(P)$ for the maximal degree of a coordinate function, noting that we will also view the coordinates of $P$ as binary homogeneous forms of degree $\hgeom(P)$.
And we set
\[\harith(P)=\sum_{v\in M_K}n_v\log\|P\|_v,\]
where $\|P\|_v=\max\|P_i\|_v$. In other words, $\harith$ is the usual Weil height on $\operatorname{Hom}_{\hgeom(P)}(\PP^1, \PP^N)\cong\PP^{(N+1)(\hgeom(P)+1)-1}$. Although we do not explicitly reference it here, our definition of arithmetic and geometric heights on $\PP^N_{K(t)}$ is motivated in part by work of Altman~\cite{MR292843}, foretelling more modern work on function fields heights (e.g., Moriwaki~\cite{MR1779799}; the language of presentations in Section~\ref{sec:B}, from Bombieri and Gubler~\cite{MR2216774}, is even closer to Altman).

It is convenient to note some basic facts on norms of polynomials (see, e.g., \cite[p.~22, p.~27]{MR2216774}).
\begin{lemma}[Gau\ss, Gelfond]\label{lem:gauss}
If $v$ is non-archimedean, then  $\|fg\|_v=\|f\|_v\|g\|_v$. In general,
\begin{equation}\label{eq:nonarchgauss}2^{-\deg(fg)}\|f\|_v\|g\|_v\leq \|fg\|_v\leq (\deg(g)+1)\|f\|_v\|g\|_v\leq 2^{\deg(fg)}\|f\|_v\|g\|_v. \end{equation}
\end{lemma}

Although $\harith(P)$ is defined in terms of coordinates in $K[t]$ chosen with no common factor, it is useful to note that some information can still be gleaned from coordinates with common factors.
\begin{lemma}\label{lem:arithgcd}
Let $P_0, ..., P_N\in K[t]$, not all zero, with greatest common factor $s$, and let $P=[P_0:\cdots :P_N]$. Then
\[\sum_{v\in M_K}n_v\log\|P\|_v=\harith(P)+h(s)+O(\hgeom(P)+\deg(s))\]
where $h(s)$ is the height of the projective tuple of coefficients of $s$.
\end{lemma}

\begin{proof}
Let $s$ be the greatest common factor of $P_{0}, ..., P_{N}$, with $sQ_{i}=P_{i}$. Then
 by Lemma~\ref{lem:gauss},
\[\left|\log\|Q\|+\log\|s\|_v- \log\|P\|_v\right|\leq\max\deg(P_{i})\log^+|2|_v.\]
Summing over all places, and nothing that $\hgeom(P)=\max\deg(P_i)-\deg(s)$, proves the lemma.
\end{proof}

Write write $P_t$ for $P$ evaluated at $t\in\PP^1_K$. It follows from the usual facts about heights that \[h(P_t)=\hgeom(P)h(t)+O_P(1),\] and the main lemma of this section is an explicit estimate on the error term. Its proof hinges  on effective elimination of variables, in the vein of the effective Nullstellensatz of Masser and W\"{u}stholz~\cite{MR704399} (see also more recents results of Krick, Pardo, and Sombra~\cite{MR1853355}).
 These results turn out to be more convenient to apply in spirit than in letter, however, given our conventions and normalizations, and so we work mostly with the constituent parts.

We will need the following fact from linear algebra, which is essentially~\cite[Lemma~4]{MR704399} stated slightly more generally (it is also just Cramer's Rule, and a proof can be found in the arXiv version of this paper). By an \emph{$r\times r$ signed minor} of a matrix $A$ with entries from some commutative ring with identity, we mean $\pm 1$ times the determinant of some $r\times r$ submatrix of $A$.
\begin{lemma}[Cramer's Rule]\label{lem:cramer}
	Let $R$ be an integral domain, and suppose that we have a homogeneous system of linear equations over $R$ in $x_1, ..., x_p$ of rank $r$, which admits a solution with $x_s\neq 0$. Then there is a solution with $x_s\neq 0$ in which each $x_j$ is an $r\times r$ signed minor of the  coefficient matrix.
\end{lemma}

\begin{lemma}\label{lem:explicitspec}
	Let $P\in \PP^N(K(t))$. Then
	\begin{multline*}
-\harith(P)-\log(\hgeom(P)+1) \\ \leq \hgeom(P)h(t)-h(P_t)\\ \leq  4\hgeom(P)\harith(P)+4\hgeom(P)\log \hgeom(P)\\+8\hgeom(P)\log 2+\log\hgeom(P)+\log (N+1)
\end{multline*}
if $\hgeom(P)\neq 0$. (If $\hgeom(P)=0$, then $h(P_t)=\harith(P)$.)
\end{lemma}

\begin{proof}
Write $P$ as a tuple of homogeneous forms of degree $\hgeom(P)$ in $t=[t_0: t_1]$. In one direction we have from the triangle inequality that
\begin{align*}
	\log\|P_t\|&\leq \hgeom(P)\log\|t\|_v+\log\|P\|_v+\log^+|\hgeom(P)+1|_v,
\end{align*}
and summing over all places gives
\[h(P_t)\leq \hgeom(P)h(t)+\harith(P)+\log(\hgeom(P)+1).\]

On the other hand, consider the system of equations
\begin{align}
a t_0^{2\hgeom(P)-1}&=P_0A_{0, 0}+\cdots +P_NA_{0, N}\label{eq:elim1}\\	
 a t_1^{2\hgeom(P)-1}&=P_0A_{1, 0}+\cdots +P_NA_{1, N}\label{eq:elim2}
\end{align}
to be solved with $a\in K$ and $A_{i, j}\in K[t_0, t_1]$. By the Nullstellensatz over $K$, there is a solution with $a=1$ if we replace the exponent on the left-hand-side by something sufficiently large, and the fact that we might take the given exponent follows from, e.g., \cite[Theorem~1.1]{MR2198324} (although in this case one can also prove the existence of a solution just by linear algebra; see the appendix to the arXiv version). Identifying coefficients of monomials in $t_0, t_1$ on both sides of each equation, we then have a system of linear equations in $a$ and the coefficients of the various $A_{i, j}$, which has a solution in $K$ with $a\neq 0$. By Lemma~\ref{lem:cramer}, there is a solution with $a\neq 0$, and in which\ $a$ and the coefficients of the $A_{i, j}$ are all $r\times r$ minors of a matrix whose entries are coefficients of the $P_i$, where $r$ is the rank of the system. From this,
\[\log\|A_{i, j}\|_v\leq r\log\|P\|_v+\log^+|r!|\]
and
\begin{align*}
\log|a|_v+(2\hgeom(P)-1)\log\|t\|_v&\leq \log\|P_t\|_v+\log\|A_{i, j}(t)\|_v+\log^+|N+1|_v\\
&\leq \log\|P_t\|_v+(\hgeom(P)-1)\log\|t\|_v\\& \quad +\log^+|\hgeom(P)|_v+r\log\|P\|_v +\log^+|r!|_v\\&\quad+\log^+|N+1|_v\\
\log|a|_v+\hgeom(P)\log\|t\|_v&\leq \log\|P_t\|_v+r\log\|P\|_v+r\log^+|r|_v\\&\quad+\log^+|N+1|_v+\log^+|\hgeom(P)|_v,
\end{align*}
since $r!\leq r^r$ for all $r\geq 1$.
Summing over all places,
\[\hgeom(P)h(t)\leq h(P_t)+r\harith(P)+r\log r+\log\hgeom(P)+\log (N+1).\]

It remains to bound $r$, which is the rank of the system of linear equations satisfied by $a$ and the coefficients of the $A_{i, j}$.
Equations~\eqref{eq:elim1} and~\eqref{eq:elim2} each involve $2\hgeom(P)$ monomials in $t_0, t_1$, and hence the resulting linear system contains at most $4\hgeom(P)$ equations; we thus have $r\leq 4\hgeom(P)$, and so
\begin{align*}
\hgeom(P)h(t)&\leq h(P_t)+4\hgeom(P)\harith(P)+4\hgeom(P)\log (4\hgeom(P))\\&\quad+\log\hgeom(P)+\log (N+1).
\end{align*}
\end{proof}

Now fix a morphism $f:\PP^N\to \PP^N$ of degree $d$ over $K(t)$. That is, $f$ is given by $N+1$ homogeneous forms of degree $d$, whose coefficients are polynomials in $t$. We set $\hgeom(f)$ to be the maximum of these degrees, and $\harith(f)$ to be the height of the grand tuple of coefficients of the $f_i$. The following lemma estimates the difference between the canonical height and the usual Weil height in our families, and is an explicit version of a result already appearing in~\cite{MR1255693}.

\begin{lemma}\label{lem:famheights}
We have
\begin{equation}\label{eq:genfib}|\hgeom(P)-\hat{h}_f(P)|\leq C_{1}\end{equation}
and (for all but finitely many $t$)
\begin{equation}\label{eq:fibreheight}|h(Q)-\hat{h}_{f_t}(Q)|\leq C_{1}h(t)+C_2\end{equation}
with
\[C_3=(N+1)^2(N(d-1)+1)^N\]
\[C_{1}=\frac{C_3\hgeom(f)}{d-1},\]
and
\begin{multline*}
C_2=\frac{C_3(\harith(f)+\log(\hgeom(f)+1)+\log C_3)}{d-1}\\+\frac{\log(N+1)+N\log (N(d-1)+1)}{d-1}.
\end{multline*}
\end{lemma}

\begin{proof}
Again, one direction is straightforward by the triangle inequality. Let $f$ be given by homogeneous forms $F_i$, whose coefficients are homogeneous forms of degree  $\hgeom(f)$ in $t=[t_0:t_1]$.
Then
\begin{equation}\label{eq:genhupper}\hgeom(f(P))\leq \max\deg(F_i(P_0, ..., P_N))\leq d\hgeom(P)+\hgeom(f).\end{equation}

Similarly, we have
for $Q\in \PP^N(K)$,
\begin{align*}
\log\|F_i(Q)\|_v& \leq d\log\|Q\|_v+\log\|F_{i, t}\|_v+N\log^+|d+1|_v\\
&\leq 	d\log\|Q\|_v+\hgeom(f)\log\|t\|_v+\log\|F\|_v+\log^+|\hgeom(f)+1|_v\\ &\quad+N\log^+|d+1|_v,
\end{align*}
so
\begin{equation}
h(f_t(Q))\leq dh(Q)+\hgeom(f)h(t)+\harith(f)+\log(\hgeom(f)+1)+N\log(d+1).\label{eq:fibhupper}	
\end{equation}

On the other hand, consider the system of equations 
\begin{align}
a X_0^{e}&=F_{0}(\mathbf{X})A_{0, 0}(\mathbf{X})+\cdots +F_N(\mathbf{X})A_{N, 0}(\mathbf{X})\nonumber\\
&\vdots\label{eq:nullres} \\
a X_N^{e}&=F_{0}(\mathbf{X})A_{N, i}(\mathbf{X})+\cdots +F_N(\mathbf{X})A_{N, N}(\mathbf{X})\nonumber
\end{align}
with $e=(N+1)(d-1)+1$,
which we hope to solve with $a\in K(t)$ and $A_{i, j}(\mathbf{X})\in K(t)[\mathbf{X}]$. By Macaulay's work on resultants~\cite{MR1577000} (see Lang~\cite[Lemma~3.7, p.~394]{MR1878556} for a more recent treatment), there is a solution with $a$ the resultant of the homogeneous forms $F _i(\mathbf{X})$, which will be nonzero as $f$ is a morphism. Then, viewing the system as a system of linear equations in the coefficients of the monomials in $\mathbf{X}$, there is a solution with $a\neq 0$ and such that $a$ and the coefficients of the $A_{i, j}$ are all $r\times r$ minors of some matrix whose entries are among the coefficients of the $F_{i}$ (where $r$ is the rank of the resulting system). So
\begin{align*}
\deg(a)+e\deg(P_i)&\leq \deg(F_i(P_0, .., P_N))+\max\deg(A_{i, j}(P_0, ..., P_N))\\
&\leq \deg(F_i(P_0, ..., P_N))+r\hgeom(f)+(e-d)\deg(P),
\end{align*}
since the coefficients of the $A_{i, j}$ are $r\times r$ minors of a matrix whose entries are coefficients of $f$.
From this,
\[\deg(a)+d\hgeom(P)\leq \max\deg(F_i(P_0, ..., P_N))+r\hgeom(f).\]
Of course, the homogeneous forms $F_i(P_0, ..., P_N)$ might have a common factor, but it is a divisor of $a$, and hence even after eliminating this we have
\begin{equation}\label{eq:genhlower}d\hgeom(P)\leq \hgeom(f(P))+r\hgeom(f).\end{equation}
It then suffices to bound $r$. Since the number of  monomials of degree $D$ in $N+1$ homogeneous variables is no greater than $(D+1)^N$, the system of linear equations implied by~\eqref{eq:nullres} consists of at most $(N+1)((N+1)(d-1)+2)^N$ equations in at most $1+(N+1)^2(N(d-1)+1)^N$ unknowns, so we may take $r\leq C_3$ (the system has a nontrivial solution, and so the rank is strictly less than the number of variables). Combining~\eqref{eq:genhupper} with~\eqref{eq:genhlower}, we obtain
\[\left|d\hgeom(P)-\hgeom(f(P))\right|\leq C_3\hgeom(f).\]

On the other hand, plugging in $t$ with $a(t)\neq 0$ gives
\begin{equation}\label{eq:thishere}\log|a(t)|_v+e\log\|Q\|_v\leq \max|F_{i, t}(Q)|_v+\max|A_{i,j}(Q)|_v+\log^+|N+1|_v\end{equation}
Now, each coefficient $F_{i, m}$ of $F_i$ satisfies
\[\log|F_{i, m, t}|_v\leq \hgeom(f)\log\|t\|+\log\|F\|_v+\log^+|\hgeom(f)+1|_v,\]
and each coefficient $A_{i, j, m}$ of $A_{i, j}$ is an $r\times r$ determinant of such values, so these coefficients satisfy
\[\log|A_{i, j, m}(t)|_v\leq r\log\max|F_{i, m, t}|_v+r\log^+|r|_v,\]
whence
\begin{multline}\label{eq:thatthere}
\log|A_{i, j}(Q)|_v\leq (e-d)\log\|Q\|_v +N\log^+|e-d+1|_v+r\hgeom(f)\log\|t\|+r\log\|F\|_v\\+r\log|\hgeom(f)+1|_v+r\log^+|r|_v,	
\end{multline}
for each $i, j$. Combining~\eqref{eq:thishere} and~\eqref{eq:thatthere}, and summing over all places of $K$, we obtain 
\begin{multline*}
dh(Q)\leq h(f_t(Q))+\log(N+1)+N\log (N(d-1)+1)+r\hgeom(f)h(t)+r\harith(f)\\+r\log(\hgeom(f)+1)+r\log r,
\end{multline*}
as long as $a(t)\neq 0$. We bound $r$ above by $C_3$, and then note that this error term is larger than that in~\eqref{eq:fibhupper}, and so even in absolute value $dh(Q)-h(f_t(Q))$ is bounded by this error term.

The final claims follow from a standard telescoping sum argument: If $S$ is any set, $\phi:S\to S$ is any function, and $\psi:S\to\RR$ is a non-negative function satsifying
\[|d\psi(x)-\psi\circ \phi(x)|\leq C,\]
then for $\hat{\psi}(x)=\lim_{k\to\infty} d^{-k}\psi\circ \phi^k(x)$, we have
\[|\hat{\psi}(x)-\psi(x)|\leq \frac{C}{d-1}.\]
\end{proof}

The previous lemma gives a good indication of how the geometric height grows in an orbit. The next lemma recapitulates this, and gives us some estimate on the more subtle arithmetic height.

\begin{lemma}\label{lem:growth}
Let $P\in \PP^N(K(t))$ and let $f:\PP^N\to \PP^N$ over $K(t)$. Then
\begin{equation}\label{eq:trivboundgeom}\hgeom(f^k(P))\leq d^kC_4.\end{equation}
and 
\begin{equation}\label{eq:trivboundarith}\harith(f^k(P))\leq d^kC_5.\end{equation}
for
\[C_4=\hgeom(P)+\frac{1}{d-1}\hgeom(f)	,\]
\begin{multline*}
C_5=\harith(P)+\hgeom(P)\log 2+\frac{1}{d-1}\harith(f)+\frac{\log 2}{d-1}\hgeom(f)\\+\frac{N}{d-1}\log(d+1)
+\frac{d\log d}{(d-1)^2}+\frac{d\log 2 +d\log^+(\hgeom(P)+\frac{1}{d-1}\hgeom(f))}{(d-1)}.
\end{multline*}
\end{lemma}

\begin{proof}
The claim~\eqref{eq:trivboundgeom} comes from the previous lemma, or specifically from~\eqref{eq:genhupper} combined with the usual telescoping sum argument.

For the second claim, choose again homogeneous forms $F_i$ representing $f$, and let $P_0=P$ with entries $P_{0, 0}, ..., P_{0, N}$. We now define a sequence of tuples $P_k$ of polynomials $P_{k, i}$ by
\[P_{k+1, i}=F_i(P_{k}).\]
Writing $F_{i, \mathfrak{m}}\in K(t)$ for the coefficient of monomial $\mathfrak{m}$ in $F_i$, and noting that there are at most $(d+1)^N$ monomials of degree $d$ in the $P_i$, we have for $v\in M_K$  that
\begin{align}
\log\|P_{k+1, i}\|_v&\leq \max\log\|F_{i, \mathfrak{m}}\mathfrak{m}(P_k)\|_v+N\log^+|d+1|_v\nonumber\\
&\leq \log\|F\|_v+d\log\|P_k\|_v+d\log^+|\max\deg(P_{k,i})+1|_v\label{eq:pkfi}\\&\quad +N\log^+|d+1|_v \nonumber
\end{align}
by Lemma~\ref{lem:gauss}.

We digress briefly to note that, for any real number $x\geq 0$, we have
\begin{equation}\label{eq:weirdsum}\sum_{j=0}^k\frac{\log(1+d^jx)}{d^j}\leq \frac{d\log d}{(d-1)^2}+\frac{d(\log 2+\log^+ x)}{(d-1)}.\end{equation}
To see this, simply note that $\log(1+d^j x)\leq \log 2 + j\log d+\log^+ x$, and apply the usual sums
\[\sum_{j=0}^\infty d^{-j}=\frac{d}{d-1}\qquad\text{ and }\qquad\sum_{j=0}^\infty jd^{-j}=\frac{d}{(d-1)^2}.\]

 Using the fact that 
\[\max\deg(P_{k, i})\leq d^k(\hgeom(P)+\frac{1}{d-1}\hgeom(f))\] and the estimate~\eqref{eq:weirdsum},
we may iterate~\eqref{eq:pkfi} to obtain
\begin{multline}\label{eq:pkest}
	\log\|P_k\|_v\leq d^k\log\|P_0\|_v+\frac{d^k-1}{d-1}\left(\log\|F\|_v+N\log^+|d+1|_v \right) \\ +\frac{d\log^+|d|_v}{(d-1)^2}+\frac{C_v}{(d-1)}
\end{multline}
where
\[C_v=\begin{cases} d\log 2+d\log^+ (\hgeom(P)+\frac{1}{d-1}\hgeom(f))& v\text{ is archimedean}\\
0 &\text{ otherwise.}	
\end{cases}
\]
When we sum over all places, we obtain
\begin{multline*}
\sum_{v\in M_K}n_v\log\|P_k\|_v\leq d^k\harith(P)+\frac{d^k-1}{d-1}\left(\harith(f)+N\log(d+1)\right) \\ +\frac{d\log d}{(d-1)^2}+\frac{d\log 2+d\log^+(\hgeom(P)+\frac{1}{d-1}\hgeom(f))}{(d-1)},	
\end{multline*}
although we note that it is entirely possible that $P_{k, 0}, ..., P_{k, N}$ admit a common factor, and so the right-hand-side is not necessarily $\harith(f^k(P))$. By the proof of Lemma~\ref{lem:arithgcd}, though, we have
\begin{align*}
\harith(f^k(P))&\leq \sum_{v\in M_K}n_v\log\|P_k\|_v+\max\deg(P_i)\log 2\\ &\leq \sum_{v\in M_K}n_v\log\|P_k\|_v+d^k(\hgeom(P)+\frac{1}{d-1}\hgeom(f))\log 2.
\end{align*}
Combining with~\eqref{eq:pkest}, we have the claimed bound~\eqref{eq:trivboundarith}, since $d^k\geq 1$.
\end{proof}

%

%
%
%

We now prove a special cases of the main result. As noted in Section~\ref{sec:B}, Lemma~\ref{lem:red}, this already implies the result in general over the base $B=\PP^1$.

\begin{proof}[Proof of Theorem~\ref{th:main} when $B=\PP^1$ and $X=\PP^N$]
First, when we apply Lemma~\ref{lem:famheights}, we will need to know that $f_t$ is an endomorphism of $\PP^N$ of degree $d$, and so we will need $a(t)\neq 0$, 
%
with $a$ as in the proof of Lemma~\ref{lem:explicitspec}. If $a(t)=0$, then $h(t)\leq h(a)+\deg(a)\log 2$. Since $a$ is an $r\times r$ determinant in the coefficients of $f$, for some $r\leq C_3$, we have $a(t)\neq 0$ as long as
\begin{equation}\label{eq:goodred}h(t)> C_3\left(\harith(f) +\hgeom(f)\log 2+\log C_3\right)\end{equation} and so we will assume this lower bound on $h(t)$. We also then have $f(P)_t=f_t(P_t)$. For now, it will be convenient also to suppose that there is no $k$ with $\hgeom(f^k(P))=0$ (this implies $C_4> 0$), although we shall see that things get even easier when this fails.

We have, for any $k$,
	\begin{align}
	\nonumber d^k\left|\hat{h}_{f_t}(P_t)-\hat{h}_f(P)h(t)\right|	&=\left|\hat{h}_{f_t}(f^k(P)_t)-\hat{h}_f(f^k(P))h(t)\right|\\
\label{eq:tgpn1}	&\leq  \left|\hat{h}_{f_t}(f^k(P)_t)-h(f^k(P)_t)\right|\\
\label{eq:tgpn2} &\quad +\left|\hgeom(f^k(P))h(t)-h(f^k(P)_t)\right|\\ 
\label{eq:tgpn3}&\quad +\left|\hgeom(f^k(P))h(t)-\hat{h}_f(f^k(P))h(t)\right|,
	\end{align}
	and we will bound the three terms separately. 
	For~\eqref{eq:tgpn1}, we have from~\eqref{eq:fibreheight} in Lemma~\ref{lem:famheights} that
	\[\left|\hat{h}_{f_t}(f^k(P)_t)-h_{\PP^N}(f^k(P)_t)\right|\leq C_{1}h(t)+C_2.\]
		For~\eqref{eq:tgpn3}, Lemma~\ref{lem:famheights}~\eqref{eq:genfib} gives
		\[\left|\hgeom(f^k(P))h(t)-\hat{h}_f(f^k(P))h(t)\right|\leq C_{1} h(t).\]

The term~\eqref{eq:tgpn2} requires more attention. By Lemmas~\ref{lem:explicitspec} and~\ref{lem:growth}, we have
\begin{align*}
\left|\hgeom(f^k(P))h(t)-h(f^k(P)_t)\right|&\leq 4\hgeom(f^k(P))\harith(f^k(P))\\ &\quad +4\hgeom(f^k(P))\log \hgeom(f^k(P))\\&\quad+8\hgeom(f^k(P))\log 2+\log\hgeom(f^k(P))\\&\quad+\log (N+1)\\
&\leq 4d^{2k}C_4C_5+4d^kC_4\log (d^kC_4)\\&\quad +8d^kC_4\log  2+\log(d^kC_4)+\log(N+1)\\
&\leq d^{2k}C_6
\end{align*}
for \[C_6=4C_4C_5+4C_4\log(dC_4) +8C_4\log  2+\log(dC_4)+\log(N+1)\] since $1, k\leq d^k$.
	Now choose $k$ so that
	\[d^{2k}\leq h(t)< d^{2(k+1)}\]
	and hence $d^{-k}\leq dh(t)^{-1/2}$,
to obtain 
	\begin{align}
\nonumber\left|\hat{h}_{f_t}(P_t)-\hat{h}_f(P)h(t)\right|&\leq dh(t)^{-1/2}\Big(2C_{1}h(t)+C_2+C_6h(t)\Big)\\
& \leq h(t)^{1/2}d(2C_{1}+C_2+C_6)\label{eq:mainb}
	\end{align}
		as long as $h(t)\geq 1$ (in addition to our previous assumed lower bound).
		
We are left to deal with the case in which $\hgeom(f^k(P))=0$ for some values of $k$. Note that, for those values of $k$, the estimates in~\eqref{eq:tgpn1} and~\eqref{eq:tgpn3} remain unchanged, while \eqref{eq:tgpn2} becomes
\[\left|\hgeom(f^k(P))h(t)-h(f^k(P)_t)\right|=h(f^k(P)_t)=\harith(f^k(P))\leq d^kC_5\]
by Lemma~\ref{lem:explicitspec} and Lemma~\ref{lem:growth}.  Since $d^k\leq d^{2k}$, we obtain by the same estimate as~\eqref{eq:mainb}  with $C_6$ replaced by $C_5$, and in general
\begin{equation}
\left|\hat{h}_{f_t}(P_t)-\hat{h}_f(P)h(t)\right|\leq h(t)^{1/2}d(2C_{1}+C_2+\max\{C_5, C_6\})\label{eq:techbound}	
\end{equation}
 even for values of $k$ with $\hgeom(f^k(P))=0$.
\end{proof}

\begin{proof}[Proof of Corollary~\ref{cor:main}]
Suppose that $P_t$ is preperiodic for $f_t$. By~\eqref{eq:techbound}, under the assumption~\eqref{eq:goodred} we have $\hat{h}_{f_t}(P_t)=0$ and hence
\[\hat{h}_f(P)h(t)^{1/2}\leq d(2C_{1}+C_2+\max\{C_5, C_6\}),\]
unless $h(t)\leq 1$.
Now,
\begin{align*}
C_5&\leq \htot(P)+\frac{1}{d-1}\htot(f)+\frac{d\log^+2(\htot(P)+\frac{1}{d-1}\htot(f))}{(d-1)}\\&\quad+\frac{N}{d-1}\log(d+1)+\frac{d\log d}{(d-1)^2}\\
&\leq (6+C_7)(\htot(P)+\htot(f))
\end{align*}
once $\htot(P)+\htot(f)\geq 1$, for
\[C_7=\frac{N}{d-1}\log(d+1)+\frac{d\log d}{(d-1)^2}.\]
And
\[C_4=\hgeom(P)+\frac{1}{d-1}\hgeom(f)\leq \htot(P)+\htot(f),\]
so
\begin{align*}
C_6&=4C_4C_5+4C_4\log(dC_4) +8C_4\log  2+\log(dC_4)+\log(N+1)\\
&\leq 4C_4C_5 +4dC_4^2+(8+d)C_4\log 2+\log(N+1)\\
&\leq 4(6+C_7+d)(\htot(P)+\htot(f))^2+(8d+2)\log2(\htot(P)+\htot(f))\\&\quad +\log(N+1)\\
&\leq (12d+26+4C_7+\log(N+1))(\htot(P)+\htot(f))^2
\end{align*}
as long as $\htot(P)+\htot(f)\geq 1$. Finally,
\[C_{1}=\frac{C_3\hgeom(f)}{d-1}\leq \frac{C_3}{d-1}(\htot(P)+\htot(f)),\]
and
\begin{multline*}
C_2\leq \frac{C_3(\htot(f)+\log C_3)}{d-1}+\frac{\log(N+1)+N\log (N(d-1)+1)}{d-1}\\
\leq \frac{2C_3^2+C_8}{d-1}(\htot(P)+\htot(f))
\end{multline*}
for
\[C_8=\log(N+1)+N\log (N(d-1)+1).\]
Combining these gives
\[\hat{h}_f(P)h(t)^{1/2}\leq C_9(\htot(P)+\htot(f))^2)\]
for
\[C_9=d\left(2\frac{2C_3+2C_3^2+C_8}{d-1}+12d+26+4C_7+\log(N+1)\right).\]
As long as $\hat{h}_f(P)\neq 0$, this gives
\[h(t)\leq C_9^2\frac{(\htot(P)+\htot(f)+1)^4}{\hat{h}_f(P)^2},\]
still under the hypothesis that $\htot(P)+\htot(f)\geq 1$ and~\eqref{eq:goodred}, where $C_9$ depends just on $N$ and $d$. On the other hand, if $\htot(P)+\htot(f)<1$, then we have $\harith(P), \hgeom(P), \harith(f), \hgeom(f)<1$, in which case $d(2C_{1}+C_2+\max\{C_5, C_6\})$ is already bounded just in terms of $N$ and $d$. So, increasing $C_9$ if necessary, we still have the bound above.

Of course, if~\eqref{eq:goodred} fails, then we have the alternative bound 
\[h(t)\leq C_3\left(\harith(f) +\hgeom(f)\log 2+\log C_3\right)\leq C_3\log C_3(\htot(f)+1).\]
The Corollary follows by taking $C=\max\{C_9^2, C_3\log C_3\}$.
\end{proof}


\section{More general canonical heights}\label{sec:k3}

Here we continue working over the field $K(t)$, with $K$ a number field, but consider a more general framework. The main result here is the following.
\begin{theorem}\label{th:matrix}
Let $f:X\to X$ be a family of dynamical systems defined over $K(t)$, let $M\subseteq \operatorname{Pic}(X)$ be a free module of finite rank generated by semi-ample line bundles, with $f^*M\subseteq M$, and let $\rho(f^*)$ be the spectral radius of $f^*$ on $M$.  If $L\in M\otimes\RR$ satisfies $f^*L\cong L^{\otimes \alpha}$ with $\alpha>1$ real, then we have
\[\hat{h}_{f_t, L_t}(P_t)=\hat{h}_{f, L}(P)h(t)+O\left(h(t)^{1+\epsilon-\log\alpha/2\log\rho(f^*)}\right)\]
for any $\epsilon>0$. In particular, if $\alpha=\rho(f^*)$, we have
\begin{equation}\label{eq:k3}\hat{h}_{f_t, L_t}(P_t)=\hat{h}_{f, L}(P)h(t)+O\left(h(t)^{\frac{1}{2}+\epsilon}\right).\end{equation}
\end{theorem}

We begin by extending somewhat the machinery of the arithmetic heights defined in the previous section.
Let $X$ be a projective variety defined over $K(t)$. 
For any morphism $\phi:X\to\PP^N$ over $K(t)$  and $P\in X(K(t))$, set \[\hgeom_\phi(P)=\hgeom(\phi(P))\qquad\text{ and }\qquad\harith_\phi(P)=\harith(\phi(P)).\] 
We collect some basic properties of heights that we will need below (note that $\hgeom_\phi$ is just the usual function field height with respect to $\phi^*\mathcal{O}(1)$).
\begin{lemma}\label{lem:heightsmatrix}
	Let $\phi:X\to \PP^n$ and $\psi:X\to \PP^m$, both defined over $K(t)$.
	\begin{enumerate}
	\item \label{it:segrearith}
If $\phi\#\psi:X\to \PP^{(n+1)(m+1)-1}$ is the composition of $\phi$ and $\psi$ through the Segre embedding, then
	\[\hgeom_{\phi\#\psi}=\hgeom_\phi+\hgeom_\psi\]
	and
	\[\harith_{\phi\#\psi}=\harith_\phi+\harith_\psi+O(\hgeom_\phi+\hgeom_\psi).\]
	\item \label{it:nullarith}	
If $\phi^*\mathcal{O}(1)\cong \psi^*\mathcal{O}(1)$, then \[\hgeom_\phi=\hgeom_\psi+O(1)\] and
	\[\harith_\phi=\harith_\psi+O(\hgeom_\phi+\hgeom_\psi).\]
	\item\label{it:nonnegarith} If $\phi^*\mathcal{O}(1)\otimes \psi^*\mathcal{O}(-1)$ is generated by global sections, then
\[\hgeom_\psi\leq \hgeom_\phi + O(1),\]
and
\[\harith_\psi\leq \harith_\phi +O(\hgeom_\phi).\]
	\end{enumerate}
	\end{lemma}
	
	\begin{proof}
		The proofs for the statements about $\hgeom$, which is the usual function field height, are standard (see, e.g.,~\cite[Chapter~2]{MR2216774}). For the arithmetic height, we proceed by similar arguments.

		Note that claim~\ref{it:segrearith} is just a claim about the Segre map. Write $Q=[Q_{0}: \cdots : Q_{n}]\in \PP^n$ with $Q_{ i}\in K[t]$ with no common factor, and similarly for $Q'\in \PP^m$. If $\sigma:\PP^n\times \PP^m\to \PP^{(n+1)(m+1)-1}$ is the Segre embedding, then $\sigma(Q, Q')$ has coordinates  $Q_{i}Q'_{ j}$, which again have no common factor. By Lemma~\ref{lem:gauss},
		\begin{align*}
\left|\log\|Q_{i}\|_v+\log\|Q'_{j}\|_v-\log\|Q_{i}Q'_{j}\|_v\right|&\leq (\deg(Q_{i})+\deg(Q'_{j}))\log^+|2|_v\\
&\leq (\hgeom(Q)+\hgeom(Q'))\log^+|2|_v.	
		\end{align*}
		Taking the maximum over all $i$ and $j$, and summing over all places gives
		\[\harith(\sigma(Q, Q'))=\harith(Q)+\harith(Q')+O(\hgeom(Q)+\hgeom(Q')).\]
		Now apply this with $Q=\phi(P)$ and $Q'=\psi(P)$, as $\phi\#\psi(P)=\sigma(\phi(P), \psi(P))$ by definition.

		For claim~\ref{it:nullarith} we first consider the case where $\phi$ and $\psi$ are closed embeddings such that the natural map $H^0(\PP^n, \mathcal{O}(1))\to H^0(X, \phi^*\mathcal{O}(1))$ is surjective, and similarly for $\psi$, following the proof of~\cite[Proposition~2.5.9, pp.~49-51]{MR2216774}. In this case,  $H^0(X, \psi^*\mathcal{O}(1))$ has a basis among the $\psi_j$, and  the isomorphism $\phi^*\mathcal{O}(1)\cong\psi^*\mathcal{O}(1)$ ensures that each $\phi_i$ the $\psi_j$, say
		\[c_i\phi_i=a_{i, 0}\psi_0+\cdots +a_{i, s}\psi_m\]
		 with $c_i, a_{i, 0}\in K[t]$. Choosing a trivialization of the line bundle and coordinates $\psi_j(P)$ with no common factor, it follows that
		 \begin{align*}
\log\|\phi_i(P)\|_v&\leq \log\max\|\psi_j(P)\|_v+\log\max\|a_{i, j}\|_v-\log\|c_i\|_v\\
		 &\quad +O_v(\max\deg(\psi_j(P))+\max\deg(a_{i, j})+\deg(c_i))		\\
		 &\leq \log\|\psi_0(P), ..., \psi_m(P)\|_v+O_v(\hgeom_\psi(P))+O_v(1), 	
		 \end{align*}
		 where the $O_v(1)$ vanishes for all but finitely many $v$.
	Summing over all places of $K$,  we have
	\[\harith_\phi(P)\leq \harith_\psi(P)+O(\hgeom_\phi(P)+\hgeom_\psi(P))\]
	from Lemma~\ref{lem:arithgcd},
and by symmetry we conclude claim~\ref{it:nullarith} in this case.

In general, if $\phi$ and $\psi$ are embeddings then, as in the proof of 	\cite[Proposition~2.5.9, pp.~49-51]{MR2216774}, there exists a $k$ such that $H^0(\PP^n, \mathcal{O}(k))\to H^0(X, \phi^*\mathcal{O}(k))$ is surjective.
Writing $\phi^{\# k}$ for the composition of $\phi$ with the $k$th monomial map $\PP^n\to\PP^{\binom{n+k}{k}-1}$, and applying the previous argument to $\phi^{\# k}$ and $\psi^{\# k}$, we conclude~\eqref{it:nullarith} without the added assumption (using part~\ref{it:segrearith} to deduce that $\harith_{\phi^{\# k}}=k\harith_\phi+O(\hgeom_\phi)$).
		
		For claim~\ref{it:nullarith} in general, note that we can write $\phi^*\mathcal{O}(1)=\sigma^*\mathcal{O}(1)\otimes \tau^*\mathcal{O}(-1)$ and $\psi^*\mathcal{O}(1)\cong \theta^*\mathcal{O}(1)\otimes \xi^*\mathcal{O}(-1)$, where $\sigma, \tau, \theta, \xi$ are all embeddings. From the paragraph above and part~\ref{it:segrearith}, we have
		\[\harith_{\sigma}+\harith_\xi=\harith_{\theta}+\harith_{\tau}+O(\hgeom_\sigma+\hgeom_\xi+\hgeom_\tau+\hgeom_\theta),\]
		but also
		\[\harith_\sigma = \harith_\phi+\harith_\tau + O(\hgeom_\sigma + \hgeom_\phi+\hgeom_\tau)\]
	and
		\[\harith_\theta = \harith_\psi+\harith_\xi + O(\hgeom_\theta + \hgeom_\psi+\hgeom_\xi).\]
		These three relations prove the claim, once we note that
		\[\hgeom_\sigma = \hgeom_\phi+\hgeom_\tau + O(1)\]
	and
		\[\hgeom_\theta = \hgeom_\psi+\hgeom_\xi + O(1).\]

		For claim~\ref{it:nonnegarith}, note that if $\phi^*\mathcal{O}(1)\otimes \psi^*\mathcal{O}(-1)$ is generated by global sections, then there is  a morphism $\theta:X\to \PP^k$ such that \[\phi^*\mathcal{O}(1)\cong \psi^*\mathcal{O}(1)\otimes\theta^*\mathcal{O}(1)\cong (\psi\#\theta)^*\mathcal{O}(1).\]
		Applying parts~\ref{it:nullarith} and~\ref{it:segrearith}, and noting that $\harith_\theta\geq 0$, we have
		\begin{align*}
			\harith_\psi &\leq \harith_\psi+\harith_\theta \\
			 &= \harith_\phi + O(\hgeom_\phi + \hgeom_\psi+\hgeom_\theta)\\
			&= \harith_\phi + O(\hgeom_\phi),
		\end{align*}
		since $\hgeom_\theta, \hgeom_\psi\leq\hgeom_\theta+ \hgeom_\psi=\hgeom_\phi$.
	\end{proof}

In light of this lemma, we can and will define $\hgeom_L$ and $\harith_L$ relative to $L\in \operatorname{Pic}(X)$ by choosing $\phi:X\to \PP^n$ and $\psi:X\to \PP^m$ with $L\cong \phi^*\mathcal{O}(1)\otimes \psi^*\mathcal{O}(-1)$, and setting \[\hgeom_L=\hgeom_\phi-\hgeom_\psi\qquad\text{ and }\qquad\harith_L=\harith_\phi-\harith_\psi .\] This depends on the choice of $\phi$ and $\psi$, but Lemma~\ref{lem:heightsmatrix} circumscribes the extent of this dependence.
More explicitly, we have the following lemma.
\begin{lemma}\label{lem:isomheights}
Let $E$ be ample.
\begin{enumerate}
\item \label{it:welldef} If $L\cong M$ then	 \[\hgeom_L=\hgeom_M+O(1)\]
and
\begin{equation}\label{eq:ampledom}\harith_L=\harith_M+O(\hgeom_E).\end{equation}
In particular, these relations hold for different choices of height function for $L$.
\item\label{it:additive} For any choice of height function relative to $L$ and $M$,
 \[\hgeom_{L\otimes M}=\hgeom_{L}+\hgeom_{M}+O(1)\]
 and
\[\harith_{L\otimes M}=\harith_{L}+\harith_{M}+O(\hgeom_E).\]
\item \label{it:amps} For any $L$ we have
\[\hgeom_L=O(\hgeom_E)\]
and
\[\harith_L=O(\hgeom_E+\harith_E).\]
\end{enumerate}
\end{lemma}


\begin{proof}
The  claims about the geometric height are standard, and follow directly from Lemma~\ref{lem:heightsmatrix}.

For  claim~\ref{it:welldef} for the additive height, suppose our height functions were
$\harith_L=\harith_\phi - \harith_\psi$ and $\harith_M=\harith_\sigma - \harith_\tau$, and note that we have $\phi^*\mathcal{O}(1)\otimes \tau^*\mathcal{O}(1)\cong \sigma^*\mathcal{O}(1)\otimes \psi^*\mathcal{O}(1)$.
 Then Lemma~\ref{lem:heightsmatrix} gives
\[\harith_L=\harith_M+O(\hgeom_\phi+\hgeom_\psi+\hgeom_\sigma+\hgeom_\tau),\]
and we may apply claim~\ref{it:amps} to the geometric heights to replace the error term with $O(\hgeom_E)$.

For claim~\ref{it:additive} for the arithmetic height, again take $\harith_L=\harith_\phi-\harith_\psi$ and $\harith_M=\harith_\sigma-\harith_\tau$, and note that $\harith_{L\otimes M}=\harith_{\phi\# \sigma}-\harith_{\psi\# \tau}$ is one choice of height for $L\otimes M$. Lemma~\ref{lem:heightsmatrix} now gives
\[\harith_{L\otimes M}=\harith_L+\harith_M+O(\hgeom_\phi+\hgeom_\psi+\hgeom_\sigma+\hgeom_\tau),\]
and again the error term is $O(\hgeom_E)$.

Claim~\ref{it:amps} follows by taking $k$ large enough that $E^{\otimes k}\otimes L^{-1}$ is ample.
Then
\begin{align*}
0&\leq \harith_{E^{\otimes k}\otimes L^{-1}}+O(\hgeom_{E^{\otimes k}\otimes L^{-1}})\\	
&= k\harith_E-\harith_L+O(\hgeom_E+\hgeom_L)\\
&=-\harith_L+O(\hgeom_E+\harith_E).
\end{align*}
\end{proof}


Now, let $\mathbf{L}=(L_1, ..., L_r)$ be a tuple of line bundles generating the free module $M$ of rank $r$. We choose morphisms $\phi_i, \psi_i$  from $X$ to projective space, with $L_i\cong \phi_i^*\mathcal{O}(1)\otimes \psi_i^*\mathcal{O}(-1)$, and set
\[\hgeom_{L_i}=\hgeom_{\phi_i}-\hgeom_{\psi_i}\qquad\text{ and }\qquad \harith_{L_i}=\harith_{\phi_i}-\harith_{\psi_i}.\] Finally, for $\mathbf{x}=(x_1, ..., x_r)\in \RR^r$ we set \[\mathbf{x}^T\mathbf{L}=\bigotimes_{i=1}^r L_i^{\otimes x_i}\in M\otimes \RR,\] and
\[\hgeom_{\mathbf{x}}=\sum_{i=1}^rx_i\hgeom_{L_i}\qquad\text{ and }\qquad\harith_{\mathbf{x}}=\sum_{i=1}^rx_i\harith_{L_i}.\]
Note that, by the freeness of $M$, each element of $M\otimes \RR$ can be written uniquely as $\mathbf{x}^T\mathbf{L}$.
 Now,  suppose that $f^*L_i\cong \bigotimes_{j=1}^r L_j^{\otimes A_{i, j}}$ for all $i$, and let $A$ be the matrix with entries $A_{i, j}$, and set
\[\|A\|=\sup_{\mathbf{y}\neq \mathbf{0}}\frac{\|A\mathbf{y}\|}{\|\mathbf{y}\|},\]
noting that $\|A\|\leq r\max|A_{i, j}|$.

\begin{lemma}\label{lem:matrixpull}
Let $P\in X(K(t))$ and $\mathbf{x}\in \RR^r$, and fix $E$ ample. We have
\[\hgeom_{\mathbf{x}}(f(P))=\hgeom_{A\mathbf{x}}(P)+O(\|\mathbf{x}\|)\]
and	
\[\harith_{\mathbf{x}}(f(P))=\harith_{A\mathbf{x}}(P)+O(\|\mathbf{x}\|\hgeom_E(P)),\]
where the implied constants  depend on $f$, $A$, and $E$, but not on $P$ or $\mathbf{x}$.
\end{lemma}

\begin{proof}
If $\mathbf{e}_i$ is the $i$th standard basis vector, we have by Lemma~\ref{lem:isomheights}
\begin{align*}
\hgeom_{\mathbf{e}_i}\circ f&=\hgeom_{L_i}\circ f  \\
&=\hgeom_{f^* L_i}  \\
&=\hgeom_{A\mathbf{e}_i}+O(1),
\end{align*}
since $(\phi_i\circ f)^*\mathcal{O}(1)\cong (A\mathbf{e}_i)^T\mathbf{L}$. The implied constant depends on $i$, but of course we may take a bound that works for all $i$, giving
\begin{align*}
	\hgeom_{\mathbf{x}}(f(P))&=\sum_{i=1}^rx_i\hgeom_{\mathbf{e}_i}(f(P))\\
	&=\sum_{i=1}^rx_i\left(\hgeom_{A\mathbf{e}_i}(P)+O(1)\right)\\
	&=\hgeom_{A\mathbf{x}}(P)+O(\|\mathbf{x}\|).
\end{align*}

The relation
\[f^*L_i\cong (A\mathbf{e}_i)^T\mathbf{L}\cong \bigotimes_{j=1}^rL_j^{\otimes A_{i, j}}\]
with Lemma~\ref{lem:isomheights}~\eqref{eq:ampledom} also provides
\[\harith_{\mathbf{e}_i}\circ f=\harith_{A\mathbf{e}_i}+O(\hgeom_E),\]
 the same argument as above now produces \[\harith_{\mathbf{x}}(f(P))=\harith_{A\mathbf{x}}(P)+O(\|\mathbf{x}\|\hgeom_E(P)).\]
\end{proof}

The previous lemma in hand, we estimate the growth of the geometric and arithmetic heights in orbits.

\begin{lemma}\label{lem:specradupper}
Let $\delta>0$, let $L\in M\otimes\RR$, and let $\rho(A)$ be the spectral radius of $A$. Then
\begin{equation}\label{eq:geomgrowthmat}\hgeom_L(f^k(P))=O\left( (\rho(A)+\delta)^k\right),\end{equation}
and
\[\harith_L(f^k(P))=O\left( (\rho(A)+\delta)^{k(1+\delta)}\right),\]
where the implied constants depends on $f$, $P$, $A$, $L$, and $\delta$, but not on $k$.
\end{lemma}

\begin{proof}
Write $L=\mathbf{x}^T\mathbf{L}$. 
From Lemma~\ref{lem:matrixpull}, we have
\begin{align*}
\hgeom_{\mathbf{x}}(f^k(P))&=\hgeom_{A\mathbf{x}}(f^{k-1}(P))+O(\|\mathbf{x}\|)\\
&=\hgeom_{A^2\mathbf{x}}(f^{k-2}(P))+O(\|\mathbf{x}\|+\|A\mathbf{x}\|)\\
&=\hgeom_{A^k\mathbf{x}}(P)+O\left(\sum_{i=0}^{k-1}\|A^i\mathbf{x}\|\right).
\end{align*}
Now, by Gelfand's formula for the spectral radius, we have 
 $\|A^k\|\leq C(\rho(A)+\delta)^k$ for some constant $C$, independent of $k$. Because $\|A\mathbf{x}\|\leq r\|A\|\cdot\|\mathbf{x}\|$, we then have
\[\sum_{i=0}^{k-1}\|A^i\mathbf{x}\|\leq rC\|\mathbf{x}\|\sum_{i=0}^{k-1}(\rho(A)+\delta)^i\leq \frac{rC\|\mathbf{x}\|\left((\rho(A)+\delta)^k-1\right)}{\rho(A)+\delta-1}=O( (\rho(A)+\delta)^k).\]
On the other hand,  note that for $E$ ample,~\eqref{eq:ampledom} implies
\[\hgeom_{\mathbf{x}}=O(\|\mathbf{x}\|\hgeom_E)\]
for $\mathbf{x}\in\RR^r$. We then  also have
\[\hgeom_{A^k\mathbf{x}}(P)=O\left( \|A^k\mathbf{x}\|\hgeom_E(P)\right)=O\left((\rho(A)+\delta)^k\right).\]
This completes the proof of the first claim.

Next, observe that
\begin{align*}
\harith_{\mathbf{x}}(f^k(P))&=	\harith_{A\mathbf{x}}(f^{k-1}(P))+O(\|\mathbf{x}\|\hgeom_E(f^{k-1}(P)))\\
&=\harith_{A^k\mathbf{x}}(P)+O\left(\sum_{i=0}^{k-1}\|A^i\mathbf{x}\|\hgeom_E(f^{k-1-i}(P))\right).
\end{align*}
We have from~\eqref{eq:ampledom} that
\[\harith_{A^k\mathbf{x}}(P)=O\left( \|A^k\mathbf{x}\|(\harith_E(P)+\hgeom_E(P))\right)=O\left((\rho(A)+\delta)^k\right)\]
just as above. On the other hand,~\eqref{eq:geomgrowthmat} applied to $\hgeom_E$, combined again with $\|A^k\|\ll (\rho(A)+\delta)^k$, gives
\begin{align*}
\sum_{i=0}^{k-1}\|A^i\mathbf{x}\|\hgeom_E(f^{k-1-i}(P))&=O\left(\|\mathbf{x}\|\sum_{i=0}^{k-1}	\|A^k\|(\rho(A)+\delta)^{k-1-i}\right)\\
&=O\left(\sum_{i=0}^{k-1}(\rho(A)+\delta)^k\right)\\
&=O\left((\rho(A)+\delta)^{k(1+\delta)}\right),
\end{align*}
since $k<(\rho(A)+\delta)^{\delta k}$ for $k$ large enough.
\end{proof}

Now, away from a set of bounded height, each $L_i\in \operatorname{Pic}(X)$ specializes to a line bundle $L_{i, t}\in \operatorname{Pic}(X_t)$, and each $\phi_i$ specializes to a morphism $\phi_{i, t}:X_t\to\PP^{m_i}$ with $\phi_{i, t}^*\mathcal{O}(1)\cong L_{i, t}$. We define $h_{L_{i, t}}=h\circ \phi_{i, t}$, and for any $\mathbf{x}^T\in \RR$, we define
\[h_{\mathbf{x}^T, t}=\sum x_ih_{L_{i, t}}.\] For $L\in M\otimes \RR$ we define $h_{L_t}=h_{\mathbf{x}^T, t}$, where $\mathbf{x}^T\mathbf{L}=L$.
With these chosen height functions, we have the following.
\begin{lemma}\label{lem:matex}
For any $L\in M\otimes \RR$ and any $\delta>0$, 
\[\left|\hgeom_L(f^k(P))h(t)-h_{L_t}(f^k(P)_t)\right|= O\left((\rho(A)+\delta)^{k(2+\delta)}\right),\]
where the implied constant is independent of $t$ and $k$.
\end{lemma}

\begin{proof}
By our choices of height functions, it suffices to prove this result for $L=L_i$. In that case, Lemma~\ref{lem:explicitspec} gives us (off of a set of bounded height)
\begin{align*}
	\left|\hgeom_{L_i}(f^k(P))h(t)-h_{L_{i, t}}(f^k(P)_t)\right|&=\left|\hgeom(\phi_i\circ f^k(P))h(t)-h((\phi_i\circ f^k(P))_t)\right|\\
	&=O\Big(\hgeom(\phi_i\circ f^k(P))\harith(\phi_i\circ f^k(P)) \\ &\quad +\hgeom(\phi_i\circ f^k(P))^2\Big)\\
	&=O\Big(\hgeom_{L_i}(f^k(P))\harith_{L_i}(f^k(P)) \\ &\quad +\hgeom_{L_i}(f^k(P))^2\Big)\\
	&=O\left((\rho(A)+\delta)^{k(2+\delta)}\right)
\end{align*}
by Lemma~\ref{lem:specradupper}.
\end{proof}

\begin{proof}[Proof of Theorem~\ref{th:matrix}]
Let $M$ be generated by $L_1, ..., L_r$, semi-ample. Note that it suffices to prove the statement after replacing all of the $L_i$ by $L_i^{\otimes m}$ for some $m\geq 1$, and so we will assume without loss of generality that the $L_i$ are generated by global sections.

Let $\epsilon>0$, and choose $\delta>0$ so that
\[\frac{\log \alpha}{(2+\delta)\log(\rho(A)+\delta)}\geq \frac{\log\alpha}{2\log \rho(A)}-\epsilon.\]
Choose $k$ so that
\begin{equation}\label{eq:rhosandwich}(\rho(A)+\delta)^{k(2+\delta)}\leq h(t)<(\rho(A)+\delta)^{(k+1)(2+\delta)}.\end{equation}
From Lemma~\ref{lem:matex}
we have	\[\left|\hgeom_L(f^k(P))h(t)-h_{L_t}(f^k(P)_t)\right|\leq C_{10} (\rho(A)+\delta)^{k(2+\delta)}\]
for some constant $C_{10}$, and for $k$ (equivalently $h(t)$) large enough. Similarly, by~\cite[Theorem~3.1]{MR1255693} we have
\[\left|\hat{h}_{f_t, L_t}(Q)-h_{X_t, L_t}(Q)\right|\leq C_{11}h(t),\]
which we apply with $Q=f^k(P)_t$, and
\[|\hgeom_L(f^k(P))-\hat{h}_{f, L}(P)|\leq C_{12}.\]
Combining these, we have
	\begin{align*}
	 \alpha^k\left|\hat{h}_{f_t, L_t}(P_t)-\hat{h}_{f, L}(P)h(t)\right|	&=\left|\hat{h}_{f_t, L_t}(f^k(P)_t)-\hat{h}_{f, L}(f^k(P))h(t)\right|\\
	&\leq  \left|\hat{h}_{f_t, L_t}(f^k(P)_t)-h_{X_t, L_t}(f^k(P)_t)\right|\\
 &\quad +\left|\hgeom_L(f^k(P))h(t)-h_{L_t}(f^k(P)_t)\right|\\ 
&\quad +\left|\hgeom_L(f^k(P))h(t)-\hat{h}_{f, L}(f^k(P))h(t)\right|\\
&\leq (C_{11}+C_{12})h(t)+C_{10}(\rho(A)+\delta)^{k(2+\delta)}\\
&\leq C_{13}h(t)
	\end{align*}
	by~\eqref{eq:rhosandwich}.
We then have
\[\alpha^k=(\rho(A)+\delta)^{k\log \alpha/\log(\rho(A)+\delta)}> h(t)^{\log \alpha/(2+\delta)\log(\rho(A)+\delta)}(\rho(A)+\delta)^{-(2+\delta)},\]
whence
\begin{align*}
\left|\hat{h}_{f_t, L_t}(P_t)-\hat{h}_{f, L}(P)h(t)\right|&\leq \alpha^{-k}C_{13}h(t)\\
&\leq C_{13}h(t)^{1-\log \alpha/(2+\delta)\log(\rho(A)+\delta)}(\rho(A)+\delta)^{(2+\delta)}\\
&\leq C_{14} h(t)^{1-\frac{\log\alpha}{2\log\rho(A)}+\epsilon}
\end{align*}
for $h(t)$ large enough.
\end{proof}

Note that we have $\alpha>1$, and so $\log\alpha/2\log\rho(A)>0$. In particular Theorem~\ref{th:matrix} always gives an improvement
\[\hat{h}_{f_t, L_t}(P_t)=\hat{h}_{f, L}(P)h(t)+O\left(h(t)^{1-\eta}\right)\]
 over~\eqref{eq:csvar},
for some $\eta>0$ depending on the action of $f^*$ on $\operatorname{Pic}(X)$. It is possible to artificially concoct examples in which this improvement is arbitrarily slight. For example, if $f_1, f_2$ are rational function of degree $d\geq e\geq 2$, then $f=(f_1, f_2)$ is an endomorphism of $X=\PP^1\times \PP^1$ satisfying $f^* \mathcal{O}(a, b)=\mathcal{O}(da, eb)$. Applying the theorem to $L=\mathcal{O}(0, 1)$, we have $\alpha=e$ and $\rho=d$, and so $\log\alpha/2\log\rho$ may be made arbitrarily small. On the other hand, one could here apply the theorem to $M=\ZZ\mathcal{O}(0, 1)\subseteq\operatorname{Pic}(X)$ to obtain the better exponent of $\frac{1}{2}+\epsilon$, or even apply Theorem~\ref{th:main} to the projection onto the first coordinate to eliminate the $\epsilon$.

\begin{proof}[Proof of Theorem~\ref{th:k3}]
Theorem~\ref{th:k3} follows immediately from Theorem~\ref{th:matrix}. In particular, Silverman~\cite{MR1115546} considers K3 surfaces $X\subseteq\PP^2\times \PP^2$ defined by the simultaneous vanishing two multihomogenous forms of degree $(1, 1)$ and $(2, 2)$ respectively. The hyperplanes on the two copies of $\PP^2$ define two line bundles $D_1, D_2$, and there is an automorphism $f$ of $X$ with
\begin{align*}
	f^*D_1&= D_1^{\otimes 15}\otimes D_2^{\otimes -4}\\
	f^*D_2&= D_1^{\otimes 4}\otimes D_2^{\otimes -1}.
\end{align*}
For $\beta = 2+\sqrt{3}$, and $E^+=D_1^{\otimes\beta}\otimes D_2^{ -1}$ we have $f^* E^+ \cong (E^+)^{\otimes \beta^2}$, and $\hat{h}^+$ is the canonical height associated to $f$ and $E^+$. Similarly, for $E^-= D_1^{ -1}\otimes D_2^{\otimes\beta}$ we have $(f^{-1})^* E^-\cong (E^-)^{\otimes \beta^2}$, and $\hat{h}^-$ is the canonical height associated to $f^{-1}$ and $E^-$. 
\end{proof}


\section{The case of a general base}\label{sec:B}

In Section~\ref{sec:rational} we proved Theorem~\ref{th:main} in the case $B=\PP^1$ and $X=\PP^N$. The former restriction, and the fact that $\operatorname{Pic}^0(\PP^1)$ is trivial, appears to genuinely improve our bounds, but it turns out that the latter restriction is immaterial.

\begin{lemma}\label{lem:red}
If Theorem~\ref{th:main} is true in the case $X=\PP^N$ and $L=\mathcal{O}(1)$, over a given base $B$, then it is true as stated.
\end{lemma}

\begin{proof}
 Suppose, as in the theorem, that $X$, $L$, and $f$ are defined over $K(B)$, with $f^*L\cong L^{\otimes d}$. By a theorem of Fakhruddin~\cite{MR1995861} there exist  $N, n\geq 1$, an embedding $i:X\to \PP^N$, and a morphism $g:\PP^N\to \PP^N$ of degree $d$ such that $i\circ f=g\circ i$ and $i^*\mathcal{O}(1)\cong L^{\otimes n}$. Note that we have
\[nh_{X, L}(P)=h_{\PP^N, \mathcal{O}(1)}(i(P))+O(1),\]
and so the 
canonical heights satisfy
\begin{equation}\label{eq:twocan}n\hat{h}_{f, X, L}(P)=\hat{h}_{g, \PP^N, \mathcal{O}(1)}(i(P))+O(1).\end{equation}
On the other hand, the transformation relation for the canonical heights gives
\begin{multline}
n\hat{h}_{f, X, L}(P)=\frac{n}{d^k}\hat{h}_{f, X, L}(f^k(P))=\frac{1}{d_k}\left(\hat{h}_{g, \PP^N, \mathcal{O}(1)}(g^k\circ i(P))+O(1)\right)\\=\hat{h}_{g, \PP^N, \mathcal{O}(1)}(i(P))+O(d^{-k})	
\end{multline}
for any $k\geq 0$, so in fact the two canonical heights in~\eqref{eq:twocan} are exactly equal.

Then, for all but finitely many $t\in B$ we have a specialization $i_t:X_t\to \PP^N$ which is an embedding, with $i_t^*\mathcal{O}(1)=L_t^{\otimes n}$, and we may use the same argument on each fibre to conclude that 
\[n\hat{h}_{f_t, X_t, L_t}(Q)=\hat{h}_{g_t, \PP^N, \mathcal{O}(1)}(i(Q))\]
%
%
for any $Q\in X_t(\overline{K})$, including $Q=P_t$. Applying the theorem to $i(P)\in \PP^N(B)$ relative to $g$ then gives the result for $P\in X(B)$ relative to $f$.
\end{proof}

In light of the lemma above, we will restrict attention to the case $X=\PP^N$, $L=\mathcal{O}(1)$, and we take
 $B$ to be a smooth, projective curve of genus $g\geq 1$ over a number field $K$. (These arguments could also be used in the case $B=\PP^1$, if we use $2g=1$ throughout, but the conclusions would be weaker than those in Section~\ref{sec:rational}).

Since we will necessarily be somewhat pedantic about heights on $B$, our first goal is to define a ``reference height'' relative to each divisor.
By the Riemann-Roch Theorem there is, for each point $\beta\in B$,  a morphism $\phi_\beta:B\to \PP^{g}$ such that $\phi^*_\beta\mathcal{O}(1)=\mathcal{O}(2g[\beta])$. We fix one such map for each point, and for an $\RR$-Cartier divisor $D=\sum_{\beta\in B} m_\beta[\beta]$ define
\[\href_{B, D}=\frac{1}{2g}\sum_{\beta\in B}m_\beta h_{\PP^g}\circ\phi_\beta.\]
The functions $\href_{B, D}$ are thus well-defined and linear in $D$, while also satisfying
\[h_{B, D}=\href_{B, D}+O(1)\]
for any other choice of height function.
We will show below that if $m_\beta\in\ZZ$ for all $\beta\in B$, then there exist morphisms $\phi:B\to\PP^n$ and $\psi:B\to\PP^m$ such that \[\mathcal{O}(2gD)\otimes\psi^*\mathcal{O}(1)=\phi^*\mathcal{O}(1)\quad\text{ and }\quad 2g\href_{B, D}=h\circ\phi - h\circ\psi.\] So $2g\href_{B, D}$ is always the Weil height associated to a particular presentation of the divisor $2gD$, in the sense of~\cite[Chapter~2]{MR2216774}. On the other hand, note that $\href_{B, D}$ depends on the choice of $D$ as a representative of its divisor class (only up to a bounded function, but this matters for our argument). We also note that our reference heights are chosen so that $\href_{B, D}\geq 0$ whenever $D\geq 0$.

Now, we fix a morphism $f:\PP^N\to\PP^N$ defined over $K(B)$ of degree $d\geq 2$, and a $K(B)$-rational point $P\in \PP^N$. We will also choose a tuple of functions $P_i\in K(B)$ with $P=[P_0:\cdots : P_N]$, writing $P$ for the tuple of functions $P_i$ as well. Finally, we choose homogeneous forms $F_i(\mathbf{X})\in K(B)[\mathbf{X}]$ such that \[f(\mathbf{X})=[F_0(\mathbf{X}):\cdots :F_N(\mathbf{X})].\]
We write $F$ for the endomorphism of $\AA^{N+1}_{K(B)}$ given by the $F_i$, and $\overline{F}$ for the tuple of coefficients of all of the $F_i$.

\begin{lemma}\label{lem:gooddiv}
There is a finite set $S\subseteq B$ and a sequence of divisors $D_k$ on $B$ such that
\begin{enumerate}
\item $D_k$ is supported on $S$
\item $\mathcal{O}(D_k)\cong f^k(P)^*\mathcal{O}(1)$
\item $D(F, P):=\lim_{k\to\infty} d^{-k}D_k$ exists in $\operatorname{Div}(B)\otimes \RR$
\item $D(F, F(P))=dD(F, P)$
\end{enumerate}
\end{lemma}

\begin{proof}
Let $\overline{\phi}=(\phi_0, \ldots , \phi_N)$  with $\phi_i\in K(B)$. Then
for \[|\psi|_\beta=e^{-\ord_{\beta}(\psi)}\]
 set
\[D(\overline{\phi})=\sum_{\beta\in B}\log\|\phi_0, ..., \phi_N\|_\beta[\beta].\]
Note that if $\phi:B\to\PP^N$ is defined by the coordinate functions $\phi_0, ..., \phi_N$, we have
\[D(\overline{\phi})=\phi^* H_i - \operatorname{div}(\phi_i),\]
for each $i$ (where $H_i$ is the $i$th coordinate hyperplane in $\PP^N$). In particular, $\mathcal{O}(D(\overline{\phi}))\cong \phi^*\mathcal{O}(1)$.

 We assume that $T\subseteq S\subseteq B$ are finite sets large enough that $\|\overline{F}\|_\beta=1$ for all $\beta\not\in T$, and $\|P\|_\beta=1$ for all $\beta\not\in S$.
First, we have  (for all $\beta\in B$)
\[\log|F_i(P)|_\beta\leq d\log\|P\|_\beta+\log\|\overline{F}\|_\beta,\]
and so 
\begin{equation}\label{eq:divupper}D(F(P))\leq dD(P)+D(\overline{F}).\end{equation}
On the other hand, as in the proof of Lemma~\ref{lem:famheights}, there are homogeneous forms $A_{i, j}(\mathbf{X})\in K(B)(\mathbf{X})$ of degree $N(d-1)$ such that
\[X_i^{(N+1)(d-1)+1}=F_0(\mathbf{X})A_{i, 0}(\mathbf{X})+\cdots+F_N(\mathbf{X})A_{i, N}(\mathbf{X})\]
for all $i$, and so from the ultrametric inequality again we have
\[d\log\|P\|_\beta\leq \log\|F(P)\|_\beta+\log\|\overline{A}\|_\beta,\]
where $\overline{A}$ is the grand tuple of coefficients of the $A_i$. We thus have
\begin{equation}\label{eq:divlower}dD(P)\leq D(F(P))+D(\overline{A})\end{equation}
(note that we may estimate the last quantity by way of the effective Nullstellensatz, but this is not particularly useful here).

Now take $T$ large enough so that $\|\overline{A}\|_\beta=1$ for $\beta\not\in S$, which is still a finite set depending just on $P$ and $F$. From~\eqref{eq:divupper} and \eqref{eq:divlower} we see that $dD(P)-D(F(P))$ is a divisor supported on $T\subseteq S$, whose order at each point of $T$ is bounded above and below. We set $D_k=D(F^k(P))$, and note by induction that $D_k$ is supported on $S$.
Note that the entries of $F^k(P)$ define $f^k(P):B\to\PP^N$, and so $\mathcal{O}(D_k)=f^k(P)^*\mathcal{O}(1)$.

Finally, let $E$ be any divisor supported on $T$ with $D(\overline{F}), D(\overline{A})\leq E$, so that
\[-E\leq dD(P)-D(F(P))\leq E.\]
Note that $E$ need not depend on $P$, here.
By the usual telescoping sum argument,
\[-\frac{1}{(d-1)d^{\min(k, m)}}E\leq d^{-k}D(F^k(P))-d^{-m}D(F^m(P))\leq \frac{1}{(d-1)d^{\min(k, m)}}E,\]
and so $d^{-k}D_k$ converges in $\operatorname{Div}(B)\otimes\RR$ (which is just to say that the orders of $d^{-k}D_k$ at each point converge). If we set
\[D(F, P)=\lim_{k\to\infty}\frac{D_k}{d^k},\]
then $D(F, F(P))=dD(F, P)$  immediately from the definition, and from the telescoping sum we have
\begin{equation}\label{eq:sandwich}-\frac{1}{d-1}E\leq D(P)-D(F, P)\leq \frac{1}{d-1}E,\end{equation}
\end{proof}

Note that we made choices of coordinate for $F$ and $P$. If $\sigma$ and $\theta$ are any two non-zero functions on $B$, then
\[D(\sigma F, \theta P)=D(F, P)+\frac{1}{d-1}\operatorname{div}(\sigma)+\operatorname{div}(\theta),\]
and so while the construction of $D(F, P)$ is sensitive to these choices, the associated class $L(f, P)=\mathcal{O}(D(F, P))\in \operatorname{Pic}(B)$ is not. It is coherent to speculate, then, that we in fact have
\[\hat{h}_{f_t}(P_t)=h_{B, L(f, P)}(t)+O(1),\]
although that still seems out of reach.

\begin{lemma}\label{lem:hrefe}
With $D_k$, $D(F, P)$, and $E$ as in Lemma~\ref{lem:gooddiv}, we have
\[\left|\href_{B, D_k}-d^k\href_{B, D(F, P)}\right|\leq \frac{1}{d-1}\href_{B, E}.\]
\end{lemma}

\begin{proof}
It follows from~\eqref{eq:sandwich}, and the fact that the reference heights are linear in the divisor and non-negative for effective divisors, that
\[-\frac{1}{d-1}\href_{B, E}\leq \href_{D(P)}-\href_{B, D(F, P)}\leq \frac{1}{d-1}\href_{B, E}.\]
Now replace $P$ with $F^k(P)$, noting that $D(F, F^k(P))=d^kD(F, P)$, and use the linearity of the reference heights again to conclude that $\href_{B, D(F, F^k(P))}=d^k\href_{B, D(F, P)}$.
\end{proof}

 Now, as in~\cite[\S~2.5]{MR2216774}, we choose a morphism $\pi:B\to\PP^2$ which maps $B$ birationally to $\pi(B)\subseteq \PP^2$. Without loss of generality, we may assume that $\pi(B)$ is given by $F(x, y, z)=0$, for some homogeneous form $F$ of degree $\deg(F)=\deg_{\mathcal{O}{(1)}}(\pi(B))$, and with $F(0, 0, 1)=1$. We then have an isomorphism (as vector spaces) of the homogeneous coordinate ring $\mathcal{S}$ with with the space of homogeneous polynomials in $x$, $y$, $z$, with $z$-degree less than $\deg(F)$, and we identify these spaces. Recall that the tuple $\mathbf{p}=(p_0, ..., p_n)$ of elements of $\mathcal{S}$ is a \emph{presentation} of the morphism $\phi=[\phi_0:\cdots :\phi_n]:B\to \PP^n$ if and only if (1) $p_j\neq 0$ for any $j$ such that $\phi_j\neq 0$, and (2) for $j$ with $\phi_j\neq 0$, we have $p_i/p_j=\phi_i/\phi_j$ in $K(B)=K(\pi(B))$. The fact that $B$ and $\pi(B)$ have the same function field ensures that every morphism has a presentation, and given a presentation $\mathbf{p}$ we write $\deg(\mathbf{p})$ for the degree of the homogeneous forms $p_i$, and $h(\mathbf{p})$ for the height of the homogeneous tuple of coefficients of all of the $p_i$.
 
 For two morphisms $\phi:B\to \PP^n$ and $\psi:B\to \PP^m$ with coordinates $\phi_i$ and $\psi_j$, we write
 \[\phi\# \psi:B\to \PP^{(n+1)(m+1)-1}\]
 for the Segre join, with coordinates $\phi_i\psi_j$. We recall that $(\phi\#\psi)^*\mathcal{O}(1)=\phi^*\mathcal{O}(1)\otimes \psi^*\mathcal{O}(1)$, and that $h\circ (\phi\#\psi)=h\circ\phi+h\circ\psi$. If $\mathbf{p}$ and $\mathbf{q}$ are presentations of $\phi$ and $\psi$, with entries $p_i$ and $q_j$, then the tuple of homogeneous forms with entries $p_iq_j$ is a presentation of $\phi\# \psi$, and we will denote this presentation by $\mathbf{p}\#\mathbf{q}$.

The next lemma, a slight variation of~\cite[Theorem~2.5.14, p.~53]{MR2216774}, makes explicit the fact that  any height function relative to the trivial divisor class on $B$ is bounded. 
\begin{lemma}\label{lem:explicitheight}
Let $\phi:B\to\PP^n$ and $\psi:B\to 	\PP^m$ be morphisms with presentations $\mathbf{p}$ and $\mathbf{q}$ respectively, and suppose that $\phi^*\mathcal{O}(1)\cong \psi^*\mathcal{O}(1)$. Then
\[\left|h_\phi - h_\psi\right|\ll \max\{\deg(\mathbf{p}), \deg(\mathbf{q})\}^2(h(\mathbf{p})+h(\mathbf{q})+\log(1+\deg(\mathbf{p}))+\log(1+\deg(\mathbf{q})),\]
with implied constants depending only on $B$, $n$, and $m$.
\end{lemma}

\begin{proof}
This is a variant of~\cite[Theorem~2.5.14, p.~53]{MR2216774}. Specifically, fix a closed embedding $\theta:B\to \PP^3$, with presentation $\mathbf{t}$. Then \[\phi\# \theta:B\to \PP^{4n+3}\text{ and }\psi\#\theta:B\to \PP^{4m+3}\] are closed embeddings, and so we may apply \cite[Theorem~2.5.14, p.~53]{MR2216774} to obtain (for $C_{14}$ and $C_{15}$ depending just on $\pi(B)$)
\begin{align}
h_\phi-h_\psi &=	h_{\phi\# \theta}-h_{\psi\# \theta}\nonumber\\
&\leq  C_{14}(4n+4)\deg(\mathbf{q}\# \mathbf{t})^2(h(\mathbf{p}\# \mathbf{t})+h(\mathbf{p}\# \mathbf{t})+\log(1+\deg(\mathbf{p}\#\mathbf{t}))\nonumber\\&\quad +\log(1+\deg(\mathbf{q}\#\mathbf{t}))+\log(n+1)+\log 144+C_{15}).\label{eq:thetatrick}
\end{align}
On the other hand, \cite[Lemma~2.5.6, p.~48]{MR2216774} gives
\[\deg(\mathbf{p}\#\mathbf{t})=\deg(\mathbf{p})+\deg(\mathbf{t})\]
and
\[h(\mathbf{p}\#\mathbf{t})\leq h(\mathbf{p})+h(\mathbf{t})+\log(1+\deg(\mathbf{t}))+C_{16},\]
where the  constant $C_{16}$ depends only on $\pi(B)$. (Note that the lemma assumes that $\phi$ and $\theta$ are both closed embeddings but, as pointed out in  \cite[Remark~2.5.7, p.~49]{MR2216774}, that assumption is not needed for the inequalities above.) Combining with~\eqref{eq:thetatrick} above, we have
\[h_\phi-h_\psi\ll \deg(\mathbf{q})^2(h(\mathbf{p})+h(\mathbf{p})+\log(1+\deg(\mathbf{p}))+\log(1+\deg(\mathbf{q}))),\]
where the constants depend on $\pi(B)$, $n$, and our choice of $\theta$ (which may be made once for the curve $B$). The claim follows by swapping $\phi$ and $\psi$ in this bound.
\end{proof}

We now construct presentations of the morphisms $f^k(P):B\to\PP^N$.
\begin{lemma}
The morphisms $f^k(P):B\to \PP^N$  admit presentations $\mathbf{p}_k$ satisfying
\[\deg(\mathbf{p}_k),  h(\mathbf{p}_k)=O(d^k),\]
where the implied constants depend on $f$, $P$, $B$, and $N$, but not on $k$.
\end{lemma}

\begin{proof} Let $\mathbf{p}_0$ be a presentation of the morphism $P:B\to\PP^N$, whose $i$th entry we denote $p_{0, i}$. As above, we may represent $f$ as a tuple of $N$ homogeneous forms of degree $d$ with coefficients in $K(B)$, thereby associating $f$ with a point in $\PP^{\binom{d+N}{N}(N+1)-1}_{K(B)}$, i.e., a morphism $B\to \PP^{\binom{d+N}{N}(N+1)-1}$ over $K$, which admits a presentation $\mathbf{F}$. The homogeneous form in $\mathbf{F}$ corresponding to the coefficient of the monomial $\mathfrak{m}$ in the $i$th entry of $f$ will be written $F_{i, \mathfrak{m}}$.
	
	Now, for $k\geq 0$, let
	\[p_{k+1, i}=\sum_{\deg(\mathfrak{m})=d} F_{i, \mathfrak{m}}\mathfrak{m}(p_{k, 0}, ..., p_{k, N}),\]
	where the sum is over all monomials of degree $d$ in $N+1$ variables. It is easy to check by induction that $\mathbf{p}_{k}$ is a presentation of $f^k(P)$.
	
	From \cite[Lemma~2.5.6, p.~48]{MR2216774} we have
	\[\deg(\mathbf{p}_{k+1})\leq d\deg(\mathbf{p}_k)+\deg(\mathbf{F}),\]
	whence
	\[\deg(\mathbf{p}_k)\leq d^k\left(\deg(\mathbf{p}_0)+\frac{1}{d-1}\deg(\mathbf{F})\right).\]

Again from 	 \cite[Lemma~2.5.6, p.~48]{MR2216774} we have
	\[h(\mathbf{p}_{k+1})\leq dh(\mathbf{p}_k)+h(\mathbf{F})+d\log(1+\deg(\mathbf{p}_k))+C_{17},\]
	where $C_{17}$ depends only on $B$, $\pi(B)$, and $d$.
	This gives
	\begin{multline*}
	h(\mathbf{p}_k)\leq d^kh(\mathbf{p}_0)+d^k\sum_{j=0}^{k-1}d^{-j}\log\left(1+d^j\left(\deg(\mathbf{p}_0)+\frac{1}{d-1}\deg(\mathbf{F})\right)\right)\\+\frac{d^k-1}{d-1}\left(h(\mathbf{F})+C_{17}\right),
	\end{multline*}
	which provides the claim $h(\mathbf{p}_k)=O(d^k)$ as $k\to\infty$ in light of~\eqref{eq:weirdsum}, which bounds the remaining sum with no dependence on $k$.
	\end{proof}

Now let $S$ be the set of places from Lemma~\ref{lem:gooddiv}, and let $s=\#S$.
\begin{lemma}
For each $k$ there is a morphism  $\phi_{D_k}:B\to \PP^{(g+1)^s-1}$ such that  $\phi_{D_k}^*\mathcal{O}(1)=\mathcal{O}(2gD_k)$, with $h\circ \phi_{D_k}=2g\href_{B, D_k}$, and admitting a presentation $\mathbf{q}_k$ satisfying
\[\deg(\mathbf{q}_k), h(\mathbf{q}_k)=O(d^k),\]
where the implied constants depend on $f$, $P$, $B$, and $N$, but not on $k$.
\end{lemma}

\begin{proof}
Let $D$ be any effective divisor supported on $S$, set $s=\# S$,  and consider the divisor $2gD=\sum 2gm_\beta[\beta]$. For each $\beta\in S$ we have a morphism $\phi_\beta:B\to \PP^g$ with $\phi_\beta^* \mathcal{O}(1) = \mathcal{O}(2g[\beta])$, and we fix a presentation $\mathbf{p}_\beta$ for this morphism. Now, if $e_m:\PP^g\to \PP^g$ is the $m$th power map, $\phi_\beta\circ e_{m_\beta}$ is presented by $\mathbf{p}_\beta^{m_\beta}$, where powers are taken component-wise. From the estimates in the proof of~\cite[Lemma~2.5.6, p.~48]{MR2216774}, we have
	\[\deg(\mathbf{p}_\beta^{m_\beta})=m_\beta\deg(\mathbf{p}_\beta),\]
	and
	\[h(\mathbf{p}_\beta^{m_\beta})\leq m_\beta h(\mathbf{p}_\beta)+(m_\beta-1)\log(1+\deg(\mathbf{p}_\beta))+C_{18}(m_\beta-1),\]
	for $C_{18}$ depending just on $\pi(B)$. In particular, $h(\mathbf{p}_\beta^{m_\beta})\leq m_\beta C_\beta$ for some $C_\beta$ depending on $\pi(B)$ and $\beta$. Now, the morphism
	\[\phi_{D}:=\phi_{\beta_1}^{m_{\beta_1}}\# \cdots \#\phi_{\beta_s}^{m_{\beta_s}}:B\to\PP^{(g+1)^s-1}\]
 has presentation given by
 \[\mathbf{p}_{\beta_1}^{m_{\beta_1}}\# \cdots \#\mathbf{p}_{\beta_s}^{m_{\beta_s}},\]
 and the usual estimates give
 	\[\deg(\mathbf{p}_{\beta_1}^{m_{\beta_1}}\# \cdots \#\mathbf{p}_{\beta_s}^{m_{\beta_s}})=\sum \deg (\mathbf{p}_{\beta_i}^{m_{\beta_i}})\leq \deg(2gD)\max\deg(\mathbf{p}_{\beta_i})\]
 	and
 	\begin{align*}
h(\mathbf{p}_{\beta_1}^{m_{\beta_1}}\# \cdots \#\mathbf{p}_{\beta_s}^{m_{\beta_s}})&\leq \sum h(\mathbf{p}_{\beta_i}^{m_{\beta_i}})+\sum \log(1+\deg(\mathbf{p}_{\beta_i}^{m_{\beta_i}}))+C_{19}s\\ 		&\leq \deg(2gD)\max\{C_\beta\}+\sum \deg (\mathbf{p}_{\beta_i}^{m_{\beta_i}}) + C_{19}s\\
&\leq \deg(2gD)C_{20},
 	\end{align*}
 	where $C_{20}$ depends on $S$, as long as $\deg(D)\geq 1$. 	Finally, note that 
 	\begin{align*}
 	h\circ\phi_{D}&= h\circ (\phi_{\beta_1}^{m_{\beta_1}}\# \cdots \#\phi_{\beta_s}^{m_{\beta_s}})\\
 	&=h\circ\phi_{\beta_1}^{m_1}+\cdots +h\circ\phi_{\beta_s}^{m_s}\\
 	&=m_1h\circ\phi_{\beta_1}+\cdots +m_sh\circ\phi_{\beta_s}\\
 	&=2g\href_{B, D}.	
 	\end{align*}
 	The lemma follows from applying this construction to $D_k$,  since
 	\[\deg(2gD_k)=2g\hgeom(f^k(P))=O(d^k).\]
 	
 \end{proof}

\begin{lemma}\label{lem:3kbound} We have
	\begin{equation}h(f^k(P)_t)=\href_{B, D_k}(t)+O(d^{3k}),\label{eq:effheight}\end{equation}
	where the implied constants depend on $f$ and $P$, but not on $t$ or $k$.
\end{lemma}

\begin{proof} 
This is the combination of the previous few lemmas. We have \[f^k(P)^*\mathcal{O}(2g)\cong \mathcal{O}(2gD_k)\cong \phi_{D_k}^*\mathcal{O}(1),\] and also that $f^k(P):B\to \PP^N$ and $\phi_{D_k}:B\to \PP^{(g+1)^s-1}$ admit presentations of height and degree at most $O(d^k)$. The Lemma~\ref{lem:explicitheight} completes the proof.
\end{proof}

Finally, we note the following result, analogous to Lemma~\ref{lem:famheights} above.
\begin{lemma}\label{lem:canheightfibbound}
There is a divisor $D$ on $B$ such that for all but finitely many $t$,
\[\hat{h}_{f_t}=h+O(h_{B, D}(t)).\]	
\end{lemma}

\begin{proof}
	This is already in~\cite{MR1255693}, but the same proof that gives~\eqref{eq:fibreheight} in Lemma~\ref{lem:famheights} works here.
\end{proof}

\begin{proof}[Proof of Theorem~\ref{th:main} with $B$ irrational] 
As noted in Lemma~\ref{lem:red}, it suffices to treat the case $X=\PP^N$ and $L=\mathcal{O}(1)$, so assume we are in that case.   Also note that, since \begin{equation}\label{eq:cauchys}h_{B, D}=h_{B, E}+O(h_{B, E}^{1/2})\end{equation} for any two height functions of the same (positive) degree on $B$ (see~\cite[Remark~9.3.9, p.~293]{MR2216774}, and note that this is true for heights relative to $\RR$-divisors, as well), it suffices to prove the result for a particular height function. 

Let $A$ be any ample divisor on $B$, and let $h_{B, A}$ be a corresponding Weil height. Note that we have
\[h_{B, D}=O\left( h_{B, A}\right)\] as $h_{B, A}\to\infty$,
for any Weil height $h_{B, D}$ relative to any $\RR$-divisor $D$, and so in particular 
we now have (for $h_{B, A}(t)$ sufficiently large)
\[\left|\hat{h}_{f_t}(Q)-h(Q)\right|\leq C_{21}h_{B, A}(t),\]
for all $Q\in\PP^N$ by Lemma~\ref{lem:canheightfibbound},
\[|h(f^k(P)_t)-\href_{B, D_k}(t)|\leq C_{22}d^{3k}\]
by Lemma~\ref{lem:3kbound}
and
\[|\href_{B, D_k}-d^k\href_{D(F, P)}|\leq C_{23} h_{B, A}(t)\]
by Lemma~\ref{lem:hrefe}.

For $t\in B(\overline{K})$, take $k\geq 0$ with
\[d^{3k}\leq h_{B, A}(t)<d^{3(k+1)}.\]
	We then have 
\begin{align*}
\frac{h_{B, A}^{1/3}}{d}\left|\hat{h}_{f_t}(P_t)-\href_{D(F, P)}(t)\right|&\leq
d^k\left|\hat{h}_{f_t}(P_t)-\href_{D(F, P)}(t)\right|\\
&= \left|\hat{h}_{f_t}(f^k(P)_t)-d^k\href_{D(F, P)}(t)\right|\\
& \leq |\hat{h}_{f_t}(f^k(P)_t)-h(f^k(P)_t)|\\
&\quad +|h(f^k(P)_t)-\href_{B, D_k}(t)|\\
&\quad +|\href_{B, D_k}(t)-d^k\href_{D(F, P)}(t)|\\
& \leq C_{21}h_{B, A}(t)\\
&\quad +C_{22}d^{3k}\\
&\quad +C_{23}h_{B, A}(t)\\
&\leq h_{B, A}(t) (C_{21}+C_{22}+C_{23}),
\end{align*}
and hence
\[\left|\hat{h}_{f_t}(P_t)-\href_{D(F, P)}(t)\right|\leq h_{B, A}(t)^{2/3}d(C_{21}+C_{22}+C_{23}).\]

But now, $\href_{D(F, P)}$ is the Weil height relative to some $\RR$-divisor of degree $\hat{h}_f(P)$ on $B$. By~\eqref{eq:cauchys} if $h$ is a height on $B$ relative to any divisor of degree $1$,
\[\href_{D(F, P)}=\hat{h}_f(P)h(t)+O(h(t)^{1/2}),\]
and since $h_{B, A}=O( h)$ we have (since $(x+y)^\rho\leq x^\rho+y^\rho$ when $0<\rho<1$)
\[\hat{h}_{f_t}(P_t)=\hat{h}_f(P)h(t)+O(h(t)^{2/3})\]
as $h(t)\to\infty$.
\end{proof}

\bibliography{variation}
\bibliographystyle{plain}

\appendix

\section{Two facts from linear algebra}

For the arXiv version of this paper, we include two proofs from linear algebra, alluded to in Section~\ref{sec:rational}. First, an effective, special case of Hilbert's Nullstellensatz with an easy proof. This is surely known, but the author had trouble finding a reference with exactly this statement (as opposed to the \emph{a priori} slightly weaker statement that $\deg(f_iA_i)\leq 2d-1$).

\begin{lemma}
Let $k$ be an infinite field, and let $f_0, ..., f_N$ be polynomials with coefficients in $k$ and of degree at most $d$. Then there are polynomials $A_0, ..., A_N$ of degree at most $d-1$ such that
\[1=f_0A_0+\cdots + f_NA_N\]
if and only if $f_0, ..., f_N$ have no nontrivial common factor.
\end{lemma}

\begin{proof}
If the polynomials have a common non-trivial factor, then of course we cannot write 1 in this way. Also, if none of the polynomials $f_i$ have degree $d$, we can apply a previous case of the theorem, so without loss of generality $\deg(f_0)=d$.

	The equation
	\[a=f_0A_0+\cdots + f_NA_N\]
defines a homogeneous system of linear equations, obtained by identifying coefficients of the same power of the variable on both sides. There are $2d$ equations, and $(N+1)d+1$ unknowns ($a$, and the coefficients of the $A_i$). Suppose there is no solution with $a=1$, so that all solutions have $a=0$. We will show that the $f_i$ must have a common factor.

Let $B$ be a polynomial  with generic coefficients, and for any $\alpha_1, ..., \alpha_N\in k$ consider the system of equations obtained by identifying coefficients of powers of the variable in
\[a=f_0A_0+B(\alpha_1f_1+\cdots +\alpha_N f_N).\]
Any nontrivial solution to this gives a nontrivial solution to the original equation, so all solutions must still have $a=0$.
Since this new system has $2d+1$ unknowns, though, there must be a non-trivial solution. Also, since $f_0\neq 0$, and $A_0=B=0$ is the trivial solution, we must have $B\neq 0$ in any nontrivial solution, so
\[\frac{\alpha_1 f_1+\cdots + \alpha_N f_N}{f_0}=-\frac{A_0}{B},\]
where the right-hand-side is a rational function of degree at most $d-1$. It follows that $f_0$ and $\alpha_1f_1+\cdots +\alpha_N f_N$ have a nontrivial common factor (since $f_0$ has degree $d$, but the ratio has degree strictly less).

 For each of the finitely many non-trivial monic divisors $s\mid f_0$, let \[V_s=\{(\alpha_1, ..., \alpha_N)\in k^N:s\mid \alpha_1f_1+\cdots +\alpha_N f_N\},\]
 noting that $V_s$ is a linear subspace of $k^N$. By what we have shown, the $V_s$ cover $k^N$, but since $k$ is infinite, $k^N$  cannot be covered by finitely many proper subspaces. Thus we have $V_s=k^N$ for some nontrivial $s\mid f_0$, and hence $s\mid f_i$ for all $i$.
 \end{proof}

\begin{proof}[Proof of Lemma~\ref{lem:cramer}]
First note that $r\leq p$, and the case of equality is ruled out by the existence of a non-trivial solution. On the other hand, if $r=0$ then any values of the $x_j$ yield a solution, and the claim is trivially true. So we will take $1\leq r\leq p-1$.

Suppose our system of equations is
\[a_{i,1}x_1+\cdots + a_{i, p}x_p=0\]
for $1\leq i\leq q$. Since the system has rank $r$, there is some non-vanishing $r\times r$ minor of the coefficient matrix, 	and without loss of generality (permuting equations and variables) we may assume that this is the top left minor. Fix $1\leq j\leq r$, and let $C_i$ be the cofactor of the entry $a_{i, j}$, so that this minor has determinant
\[\delta = a_{1, j}C_1+\cdots +a_{r, j}C_r\neq 0.\]
Multiplying the $i$th equation by $C_i$ and summing, then, yields an equation in which the coefficient of $X_j$ is $\delta$. On the other hand, for $1\leq k\leq r$ and $k\neq j$, the coefficient of $x_k$ in this new equation is $C_1a_{1, k}+\cdots + C_ra_{r, k}=0$, because this is the determinant of a matrix with a repeated column. Finally, for $k>r$, the coefficient of $x_k$ in this new equation is
\[\delta_{j, k}=C_1a_{1, k}+\cdots +C_ra_{r, k}\] which is, up to sign,
the determinant of some other $r\times r$ submatrix of the coefficient matrix.

So from the original system we have deduced a new system of equations
\[\delta x_i +\delta_{i, r+1}x_{r+1}+\cdots + \delta_{i, p}x_p=0\quad (1\leq i\leq r),\]
where $\delta$ and the $\delta_{i, j}$ are all $r\times r$ signed minors of the coefficient matrix, and $\delta\neq 0$. This new system  has rank $r$, though, and so is equivalent to the original system.

If $r+1\leq s\leq p$, our solution is $x_s=\delta$, $x_i=-\delta_{i, s}$ for $1\leq i\leq r$, and $x_i=0$ for all other $i$. If $1\leq s\leq r$, then note that we cannot have $\delta_{s, j}=0$ for all $r+1\leq j\leq p$, or else our original equations would force $x_s=0$. So we choose some $j$ with $\delta_{s, j}\neq 0$, set $x_i=\delta_{i, j}$ for $1\leq i\leq r$,  $x_j=-\delta$, and $x_i=0$ otherwise. In any case, the $x_i$ are signed minors of the coefficient matrix, and $x_s\neq 0$.
\end{proof}

\end{document}